\documentclass[draft]{amsart}

\usepackage{amsmath}
\usepackage{amssymb}
\usepackage[all]{xy}
\usepackage{picinpar}
\usepackage{palatino}
\usepackage[usenames,dvipsnames,svgnames,table]{xcolor}
\usepackage{mathtools}

\def\eu{\mathfrak}
\def\ma{\mathbb}
\def\mc{\mathcal}

\def\p{{\mathcal P}_{\infty}}
\def\f{{\ma F}_q^{\ast}}
\def\F{{\ma F}_q}
\def\P{\mathcal P}
\def\G{\mathcal G}
\def\H{\mathcal H}
\def\Q{\mathcal Q}

\def\L{{\mathcal L}_{{\ma F}_p}}
\def\fin{\hfill\qed\bigskip}
\def\elemental#1#2#3{#1_{#3-1}#2^{p^{#3-1}}+#1_{#3-2}#2^{p^{#3-2}}+
\cdots +#1_2#2^{p^2} + #1_1#2^p+#1_0#2}
\def\elementalito#1#2#3{#1_{#3-1}#2^{p^{#3-1}}+
\cdots + #1_1#2^p+#1_0#2}
\def\matriz#1#2{\left[\begin{array}{ccccc}
#1& #1^p&\cdots&#1^{p^{n-2}}&#1^{p^{n-1}}\\
#2_2&#2_2^p&\cdots&#2_2^{p^{n-2}}&#2_2^{p^{n-1}}\\
\vdots&\vdots&\ddots&\vdots&\vdots\\
#2_{n-1}&#2_{n-1}^p&\cdots&#2_{n-1}^{p^{n-2}}&#2_{n-1}^{p^{n-1}}\\
#2_n&#2_n^p&\cdots&#2_n^{p^{n-2}}&#2_n^{p^{n-1}}
\end{array}\right]}
\def\vmatriz#1{\left[\begin{array}{ccccc}
\vec#1_1& \vec#1^p_1&\cdots&\vec#1^{p^{n-2}}_1&\vec #1^{p^{n-1}}_1\\
\vec#1_2&\vec#1_2^p&\cdots&\vec#1_2^{p^{n-2}}&\vec#1_2^{p^{n-1}}\\
\vdots&\vdots&\ddots&\vdots&\vdots\\
\vec#1_{n-1}&\vec#1_{n-1}^p&\cdots&\vec#1_{n-1}^{p^{n-2}}&
\vec#1_{n-1}^{p^{n-1}}\\
\vec#1_n&\vec#1_n^p&\cdots&\vec#1_n^{p^{n-2}}&\vec#1_n^{p^{n-1}}
\end{array}\right]}

\def\lra{\longrightarrow}
\def\Witt#1{\stackrel{_{\bullet}}{#1}}
\def\vWitt#1#2{#1^{(1)},\ldots,#1^{(#2)}\big)}
\def\v#1#2{\big({#1}_1,\ldots,{#1}_{#2}}
\def\Wittc#1#2{\v#1#2\mid\vWitt#1#2}

\newcommand{\Irr}{\operatorname{Irr}}
\newcommand{\Gal}{\operatorname{Gal}}

\newcommand{\mcd}{\operatorname{gcd}}
\newcommand{\Id}{\operatorname{Id}}
\newcommand{\Tr}{\operatorname{Tr}}
\renewcommand{\o}{\mathcal O}
\newcommand{\im}{\operatorname{im}}

\newcounter{bean}
\newcounter{2bean}

\def\l{
\begin{list}
{\rm{(\alph{bean}).-}}{\usecounter{bean}
\setlength{\labelwidth}{0.8in}
\setlength{\labelsep}{0.3cm}
\setlength{\leftmargin}{1cm}}}

\def\las{\begin{list}
	{{\rm {(\arabic{2bean})}}}{\usecounter{2bean}
\setlength{\labelwidth}{0.8in}
\setlength{\labelsep}{0.3cm}
\setlength{\leftmargin}{1cm}}}

\numberwithin{equation}{section}
\newtheorem{teorema}{Theorem}[section]
\newtheorem{proposicion}[teorema]{Proposition}
\newtheorem{lema}[teorema]{Lemma}
\newtheorem{ejemplo}[teorema]{Example}
\newtheorem{observacion}[teorema]{Remark}
\newtheorem{definicion}[teorema]{Definition}
\newtheorem{corolario}[teorema]{Corollary}

\title[Abelian $p$--extensions and additive polynomials]
{Abelian $p$--extensions and additive polynomials}
 
\author[J. Barreto]{Jonny Fernando Barreto--Casta\~neda}
\address{Departamento de Control Autom\'atico\\
Centro de Investigaci\'on y de Estudios Avanzados del I.P.N.}
\email{jbarreto@ctrl.cinvestav.mx}
 
\author[F. Jarqu\'in]{Fausto Jarqu\'in--Z\'arate}
\address{Universidad Aut\'onoma de la Ciudad de M\'exico\\
Academia de Matem\'aticas. Plantel San Lorenzo Tezonco\\
Prolongaci\'on San Isidro No. 151, Col. San Lorenzo,
Iztapalapa, C.P. 09790, M\'exico, D.F.}
\email{fausto.jarquin@uacm.edu.mx}

\author[M. Rzedowski]{Martha Rzedowski--Calder\'on}
\address{Departamento de Control Autom\'atico\\
Centro de Investigaci\'on y de Estudios Avanzados del I.P.N.}
\email{mrzedowski@ctrl.cinvestav.mx}

\author[G. Villa]
{Gabriel Villa--Salvador}
\address{
Departamento de Control Autom\'atico\\
Centro de Investigaci\'on y de Estudios Avanzados del I.P.N.}
\email{gvillasalvador@gmail.com, gvilla@ctrl.cinvestav.mx}



\date{June 7th., 2016}

\begin{document}

\begin{abstract}

In this work we present some arithmetic properties of
families of abelian $p$--extensions of global function
fields, among which are their generators and their type
of ramification and decomposition.

\end{abstract}

\maketitle

\section{Introduction}\label{S1}

The study of elementary abelian $p$--extensions has been
considered intensively by several authors. Usually, the way to
study this type of extensions is to consider all its subextensions
of degree $p$ and then apply Hasse's criterion \cite{Has34}
as well as facts known about the behavior of the primes in these
subextensions.
Once this study has been done, one returns to the total extension.
Garcia and Stichtenoth \cite{GarSti91}, changing the usual point of
view, considered those extensions that can be given by
a special type of additive polynomial; namely, the extensions
given by an equation of the form $y^q-y=\alpha$. 
In their research, they obtained the genus of these extensions and
used their results to build towers of function fields where the genus
grows much faster than the number of rational points. In the
same paper the authors mention that similar results can be 
obtained for additive polynomials having their roots in the
base field.

In this paper we consider an additive 
polynomial $f(X)$ whose roots
belong to the base field and we prove results analogous to
the ones obtained by Garcia and Stichtenoth. For instance,
given an additive polynomial, we show that any elementary
abelian $p$--extension can be described by an equation of
the type $f(X)=u$.

When the base field is a global rational function field, it is
possible to give a lower bound for the ramification index
of the ramified primes without considering its subextensions
of degree $p$. It is also possible to characterize the
fully decomposed primes. 
In the case of cyclic extensions of degree
$p$ given by an Artin--Schreier equation, the relation between
two distinct generators is well known. In this article we give
the corresponding result for elementary abelian $p$--extensions
obtained by means of additive polynomials.

In the last part of this work, we generalize the results obtained
to extensions which will be called {\em multicyclic}, and which
are given by a Witt equation of the form $\vec y^q\Witt - 
\vec y=\vec \alpha$. Much of the formalism of the elementary
abelian $p$--extensions can be translated to this new
case and in fact it is possible to keep generalizing these
results to other {\em additive vectorial polynomials} whose
roots belong to the Witt ring of the base field. In this paper
we restrict ourselves to the case of an equation of the 
type $\vec y^q\Witt - \vec y=\vec \alpha$.

\section{Notations}\label{S2}

Let $p$ a prime number, $n\in {\ma N}$ and $q=p^n$. Let $k$
be an arbitrary field of characteristic $p$. When $k$ is a
function field, $k_0$ will denote the field of constants of $k$
and we will assume that $k_0$ is a perfect field. Let
$f(X)\in k[X]$ be an additive polynomial, that is,
 $f(x+y)=f(x)+f(y)$ for every $x,y\in \bar{k}$, a fixed algebraic 
 closure of $k$. Then $f(X)$ is given by
 \begin{equation}\label{Eq2.1}
f(X)=\sum_{i=0}^n a_iX^{p^i},
\end{equation}
with $a_i\in k$. We will assume that $f$ is monic and separable,
that is, $a_n=1$ and $a_0\neq 0$. Furthermore, we will assume that
the roots of $f(X)$ belong to the base field, that is,
\begin{equation}\label{Eq2.2}
\G_f:=\{\xi\in \bar{k}\mid f(\xi)=0\}\subseteq k.
\end{equation}
As a special case, we will consider the additive polynomial
$f(X)=X^{p^n}-X=X^q-X$.

In general, when we consider polynomials of the form
$F(X)=f(X)-u\in k[X]$, we will assume that $F(X)$ is
irreducible. Let $K=k(y)$ with $f(y)=u\in k$, that is,
$F(y)=0$. Then $f(y+\xi)=f(y)+f(\xi)= f(y)=u$ for all
$\xi\in\G_f$ so that the set of roots of $F(X)$ is
\begin{gather*}
y+\G_f=\{y+\xi\mid\xi\in\G_f\}.
\end{gather*}

We have that every element $\sigma$ of $G=\Gal(K/k)$ 
is determined by $\sigma(y)$. Since $y$ and $\sigma (y)$ 
are conjugate, there exists $\xi_{\sigma}\in \G_f$ such that
$\sigma(y)=y+\xi_{\sigma}$.
From (\ref{Eq2.2}), we have that the extension
$K=k(y)$ of $k$ is a Galois extension of degree $p^n$.

\begin{proposicion}\label{P2.1} In general, for any additive
polynomial $f(X)\in k[X]$ of degree $p^n$, we have that
$\G_f$ is an additive group $\G_f\subseteq (\bar{k},+)$
isomorphic to $C_p^n=\big({\ma Z}/p{\ma Z}\big)^n$.
That is, $\G_f$ is an $\ma F_p$--vector space of dimension $n$.
\end{proposicion}

\begin{proof} For $\alpha,\beta\in\G_f$ we have
\begin{gather*}
f(\alpha+\beta)=f(\alpha)+f(\beta)=0+0=0,\\
f(0)=0,\quad f(-\alpha)=f((p-1)\alpha)=
\underbrace{f(\alpha)+\cdots +f(\alpha)}_{p-1}=
0+\cdots+0=0.
\end{gather*}

Therefore $\G_f\subseteq (\bar{k},+)$. Finally, 
$p\beta =0$ for all $\beta\in \bar{k}$ and $|\G_f|=p^n$,
so that $\G_f\cong C_p^n$.
\end{proof}

In general, if $V$ is a finite $p$--subgroup of $k$, then we
denote:
\begin{gather*}
f_V(X)=\prod_{\delta\in V}(X-\delta)
\end{gather*}
which is an additive polynomial (\cite[proof of
Proposition 13.4.10]{Vil2006}). In particular an
additive polynomial $f(X)\in k[X]$ (\ref{Eq2.1})
satisfies $f(X)=f_{\G_f}(X)$.

\begin{proposicion}\label{P2.2} With the above notation, we have
that $\theta\colon G\lra \G_f$ given by
$\theta(\sigma)=\xi_{\sigma}$, where $\sigma y=y+\xi_{\sigma}$,
is a group monomorphism so that we may consider $G\subseteq
\G_f$. When $F(X)$ is irreducible, we have the equality $G=\G_f$.
\end{proposicion}

\begin{proof}
For $\sigma,\tau\in G$, we have
\begin{gather*}
y+\xi_{\sigma\tau}=\sigma\tau(y)=
\sigma(y+\xi_{\tau})=\sigma(y)+\xi_{\tau}=y+\xi_{\sigma}+\xi_{\tau},\\
\intertext{so that}
\theta(\sigma\tau)=\xi_{\sigma\tau}=\xi_{\sigma}+\xi_{\tau}=
\theta(\sigma)+\theta(\tau).
\end{gather*}
If $\theta(\sigma)=0$, then $\sigma(y)=y$ which implies
$\sigma=\Id$. It follows that $\theta$ is a group monomorphism.

When $F(X)$ is irreducible, we have $|G|=[K:k]=\deg F(X)=
p^n=|\G_f|$. 
\end{proof}

\begin{observacion}\label{O2.3} {\rm{In case $\G_f\nsubseteq k$,
we have that the decomposition field of $f(X)$ is $k(\G_f,y)$ and
therefore $K=k(y)$ is not a normal extension of $k$. In fact we have
$\Gal(k(\G_f,y)/k)\cong \Gal(k(\G_f)/k)\ltimes\Gal(k(\G_f,y)/k(\G_f))$.
}}
\end{observacion}

In general, if $\beta_1,\ldots,\beta_m\in k$, $\L\{\beta_1,\ldots,\beta_m\}$
denotes the ${\ma F}_p$--vector space generated by $\beta_1,
\ldots,\beta_m$.

Finally, the Artin--Schreier operator is denoted by $\wp$, that is,
$\wp(c)=c^p-c$ and $\wp_a$ ($a\neq 0$)
denotes the map $\wp_a(c)=c^p-a^{p-1}c=
a^p\wp\big(\frac{c}{a}\big)$. In case of a function field
$k/k_0$, with $k_0$ a finite field,
$R_T$ denotes the polynomial ring $k_0[T]$ and
$R_T^+$ denotes the set of monic polynomials of $k_0[T]$.

The notation on Witt vectors will be given in Section \ref{S8}.

\section{The polynomial $X^q-X$}\label{S3}

In \cite{GarSti91} Garcia and Stichtenoth made a very  thorough and
complete study of extensions $K/k$ when $k$ is a function field,
$K=k(y)$ with $y^q-y=u\in k$, $\F\subseteq k_0$ and they
claim that analogous results hold for additive polynomials whose
roots belong to $k_0$. In this section we consider the case
$f(X)=X^q-X\in k_0[X]$ and recall the main result of
Garcia and Stichtenoth. 

\begin{proposicion}[Garcia and Stichtenoth, \cite{GarSti91}]\label{P3.1} 
Assume that $K/k$ is
an elementary abelian $p$--extension of degree $p^n$ and
such that ${\ma F}_{p^n} \subseteq k_0$. Then, there exists
$y\in K$ such that $K=k(y)$ whose minimal polynomial is 
$\Irr(X,y,k)=X^{p^n}-X-a$ for some $a\in k$.

Conversely, if ${\ma F}_{p^n}\subseteq k_0$ and $\varphi(X)=
X^{p^n}-X-a\in k[X]$ is irreducible, then $K=k(y)$ with $\varphi
(y)=0$ is an elementary abelian $p$--extension of degree $p^n$.
The intermediate fields $k\subseteq E_{\mu}\subseteq K$
of degree $p$ over $k$, are given by
$E_{\mu}=k(y_{\mu})$ with $\mu\in{\ma F}_{p^n}^{\ast}$ and
\[
y_{\mu}:=(\mu y)^{p^{n-1}}+(\mu y)^{p^{n-2}}+\cdots+(\mu y)^p+(\mu p),
\]
$y_{\mu}^p-y_{\mu}=\mu a$; therefore $k(y)=k(\wp^{-1}(U))$ with
$U=\{\mu a\mid \mu \in {\ma F}_{p^n}\}$. \fin
\end{proposicion}

\begin{observacion}\label{O3.3} {\rm{Garcia and Stichtenoth claim that 
in Proposition \ref{P3.1} the polynomial $X^{p^n}-X$ 
can be replaced by any monic separable
additive polynomial of degree $p^n$ whose roots belong to $k_0$
and that an analogous description can be given for elementary
abelian $p$--extensions. We will see this in Section \ref{S5}.
}}
\end{observacion}

\section{General facts on elementary abelian $p$--extensions}\label{S4}

For the time being we consider the special case of the polynomial $X^q-X$
with the notation and conventions of Section \ref{S2}. Let
$k/k_0$ be a function field, 
$G=\Gal(K/k)\cong C_p^n$, $K=k(y)$, $y^q-y=u\in k$. 

\begin{teorema}\label{T4.3} With the above notation, given a
prime divisor $\P$ of $k$, we have that there exists 
$y\in K$ such that $K=k(y)$, $y^q-y=u$ with $v_{\P}(u)\geq 0$
or $v_{\P}(u)=-\lambda p^m$ where $\lambda>0$, $\mcd(\lambda, p)=1$
and $0\leq m < n$.
If $v_{\P}(u)\geq 0$, $\P$ is not ramified in $K/k$.
If $v_{\P}(u)=-\lambda p^m$, then $p^{n-m}\mid e_{\P}$ where
$e_{\P}$ denotes the ramification index of $\P$ in $K/k$.
\end{teorema}

\begin{proof}
See the proof of Theorem \ref {T5.1}.
\end{proof}

\begin{observacion}\label{O4.1} {\rm{The number $m$ 
given in Theorem \ref{T4.3} is not unique.
}}
\end{observacion}

\begin{ejemplo}\label{Ej4.2}{\rm{
Let $k={\ma F}_q(T)$ with ${\ma F}_{p^2}\subseteq \F$ and let
$K=k(y)$ where $y^{p^2}-y =u=T^{\lambda p}$ with
$\lambda\in{\ma N}$, $\mcd(\lambda,p)=1$. 
Then $v_{\p}(T)=-\lambda p$ and in this case $m=1$, $n=2$.

Set $z:=y^p-T^{\lambda}$. Then
\begin{align*}
z^{p^2}-z&=(y^p)^{p^2}-(T^{\lambda})^{p^2}-y^p+T^{\lambda}=
(y^{p^2}-y)^p-T^{\lambda p^2}+T^{\lambda}\\
&=(T^{\lambda p})^p-
T^{\lambda p^2}+T^{\lambda}=T^{\lambda}=\nu,
\end{align*}
and in this case $v_{\p}(\nu)=-\lambda$, $m=0$ and $n=2$.

Note that necessarily $k(z)=k(y)=K$ since $e_{\p}=
p^2=[K:k]$. More generally, this is an immediate consequence
of Theorem \ref{T7.1}.
}}
\end{ejemplo}

\begin{definicion}\label{D4.4}{\rm{
When a prime $\P$ satisfies the conditions of Theorem \ref{T4.3}
with respect to the equation $f(y)=u$, we say that is in a
{\em normal form with respect to $\P$}.  From Remark \ref{O4.1} we have that
a normal form is not unique.
}}
\end{definicion}

\begin{observacion}\label{O4.5}{\rm{
It does not hold necessarily that
$e_{\P}=p^{n-m}$, see Example \ref{Ej4.2}.
In that example we have, with $m=0$, $p^2\mid e_{\P}$ and
therefore $e_{\P}=p^2$. For $m=1$, $p^{2-1}=p\mid e_{\P}$ but $e_{\P}
\neq p$. Furthermore, even in case that $m$ be the minimum with the
previous properties or even in case that $m$ is unique, it does not
necessarily follow that $e_{\P}=p^{n-m}$. However, if $m=0$, that is,
if $v_{\P}=-\lambda$, then $p^n\mid e_{\P}$, thus $e_{\P}=
p^n$ and the prime is fully ramified.
}}
\end{observacion}

\begin{ejemplo}\label{Ej4.6}{\rm{Let $k=k_0(T)$ with ${\ma F}_{p^2}
\subseteq k_0$. Let
$K=k(y_1,y_2)$ with $y_1^p-y_1=T$, $y_2^p-y_2=T^2$ if
$p>2$ and $y_2^p-y_2=T^3$ if $p=2$. Then if $\mu\in {\ma F}_{
p^2}\setminus \ma F_p$, we consider $y=y_1+\mu y_2$, and so
we have $K=k(y)$ and
\begin{align*}
y^{p^2}-y&=y_1^{p^2}+\mu^{p^2}y_2^{p^2}-y_1-\mu y_2=
(y_1^{p^2}-y_1)+\mu(y_2^{p^2}-y_2)\\
&=[(y_1^p-y_1)^p+(y_1^p-y_1)]+\mu[(y_2^p-y_2)^p+(y_2^p-y_2)]\\
&=T^p+T+\mu T^{2p}+\mu T^2=T^p(1+\mu T^p)+(T+\mu T^2)=\gamma,
\end{align*}
$v_{\p}(\gamma)=-2p$.

The other intermediate extensions $K/k$ of degree $p$ are
given by $k(y_1+\xi y_2)$, $1\leq \xi \leq p-1$ and they satisfy
\[
(y_1+\xi y_2)^p-(y_1+\xi y_2)=(y_1^p-y_1)+\xi (y_2^p-y_2)=T+\xi T^2,
\]
$v_{\p}(T+\xi T^2)=-2$. In particular $\p$ is totally ramified, 
$e_{\p}=p^2$, $n=2$, $m=1$, $n-m=1<2$.

Note that there is no way to obtain $m=0$ with any change of
variable because otherwise we would have $K=k(z)$ such that $z^{p^2}-
z=\nu\in k$ and $v_{\p}(\nu)=-\lambda$ with $\mcd(\lambda, p)=1$, 
and for any intermediate extension we would have
\[
z_{\mu}^p-z_{\mu}=\mu\nu,\quad \nu\in{\ma F}_{p^2}^{\ast}
\quad\text{and}\quad v_{\p}(\mu\nu)=-\lambda.
\]
Therefore the different exponent of ${\eu p}_{\infty}$, where
${\eu p}_{\infty}$ is a prime above $\p$, would be $(\lambda+1)(p-1)$
and $\lambda +1$ would be the unique ramification number for all
the intermediate extensions. However, in the subextensions
$k(y_1)/k$ and $k(y_2)/k$ the ramification numbers are different,
namely, $1+1=2$ and $2+1=3$ (or $3+1=4$ in case $p=2$).

Therefore $m=1$ is unique, $p^{n-m}=p^{2-1}=p\mid e_{\p}$ but
$e_{\p}= p^2\neq p$.
}}
\end{ejemplo}

\begin{ejemplo}\label{Ej4.7}{\rm{
Let $y_1^p-y_1=T$, $y_2^p-y_2=\frac{1}{T}$ and, $K=
k(y_1,y_2)$. Then $\p$ is not totally ramified since it is
unramified in $k(y_2)/k$ and in this case $e_{\p}=p$, $n=2$, $m=1$
and $p^{n-m}=p=e_{\p}$.
}}
\end{ejemplo}

\section{Additive polynomials}\label{S5}

With the notation and conventions of Section \ref{S2}, we consider
$f(X)\in k[X]$ a monic separable additive polynomial of degree $p^n$:
\begin{equation}\label{E2}
f(X)=X^{p^n}+a_{n-1}X^{p^{n-1}}+\cdots+a_2X^{p^2}+a_1 X^p +a_0 X
\in k[X],\quad a_0\neq 0.
\end{equation}

From Proposition \ref{P2.2} we know that $\G_f$ is an additive group
isomorphic to $C_p^n$ with $\G_f\subseteq \bar{k}$, that is, $\G_f
\subseteq (\bar{k},+)$.

As always, we will assume that $\G_f\subseteq k$. Let $K=k(y)$ 
with $f(y)=u\in k$. We are assuming that $F(X)=f(X)-u\in k[X]$ 
is irreducible. Thus, from Proposition \ref{P2.2} 
we have that $G=\Gal(K/k)\cong \G_f$.

Therefore, from Proposition \ref{P2.2}, with $F(X)$ irreducible, 
we obtain that
if the set $\{\varepsilon_1,\ldots,\varepsilon_n\}$
is contained in $\G_f$ then
$\{\varepsilon_1,\ldots,\varepsilon_n\}$ is a 
basis of $\G_f$ over ${\ma F}_p$ if and only if
$G=\langle \sigma_{\varepsilon_1},\ldots,\sigma_{\varepsilon_n}\rangle$.

Now, $G$ has $\frac{p^n-1}{p-1}$ subgroups of index $p$, that is,
$K/k$ has $\frac{p^n-1}{p-1}$ subextensions of degree $p$ over $k$.
We will study in more detail these subextensions.

If we denote by $z$ the elements such that $k\subseteq k(z)\subseteq
K$ with $[k(z):k]=p$, then $k(z)$ is the fixed field of a subgroup $H$
of $G$ of index $p$: $k(z)=K^H$. In this case, if $G=H\oplus
{\ma F}_p \sigma_z$, then $\Gal(k(z)/k)\cong \langle\sigma_z
\rangle$ with $\sigma_z (z)=z+1$.

In this way we have that if $E=k(z_1,\ldots,z_n)$, then
$E=K$ if and only if $G=\langle\sigma_{z_1},\ldots,\sigma_{z_n}
\rangle$. Now, let $K=k(z_1,\ldots,z_n)$ and denote $\Gal(k(z_i)/k)=
\langle\sigma_i\rangle$ and $G\cong \langle\sigma_1,\ldots,\sigma_n\rangle$.
Set $z:=\alpha_1 z_1+\cdots+\alpha_n z_n$ with $\alpha_1,\ldots,
\alpha_n\in{\ma F}_p$ not all zero. Then if $\wp(z_i)=z_i^p-z_i=
\gamma_i$, we have
\[
z^p-z=\wp(z)=\wp\Big(\sum_{i=1}^n\alpha_iz_i\Big)=\sum_{i=1}^n
\alpha_i\wp(z_i)=\sum_{i=1}^n\alpha_i\gamma_i.
\]

Note that $\{\gamma_1,\ldots,\gamma_n\}\subseteq k$ is linearly
independent over ${\ma F}_p$ since otherwise, if $\sum_{i=1}^n\alpha_i\gamma_i
=0$ with some $\alpha_{i_0}\neq 0$, then $\gamma_{i_0}=
\sum_{i\neq i_0}\alpha_{i_0}^{-1}\alpha_i\gamma_i$ so that
\begin{gather*}
\wp(z_{i_0})=z_{i_0}^p-z_{i_0}=\gamma_{i_0}=\sum_{i\neq i_0}
\alpha_{i_0}^{-1}\alpha_i\gamma_i=\sum_{i\neq i_0}\alpha_{i_0}^{-1}\alpha_i
(z_i^p-z_i)=\wp\big(\sum_{i\neq i_0}\alpha_{i_0}^{-1}\alpha_iz_i\big),
\intertext{which implies that}
\wp\big(z_{i_0}-\sum_{i\neq i_0}\alpha_{i_0}^{-1}\alpha_iz_i\big)=0.
\intertext{Hence $z_{i_0}-\sum_{i\neq i_0}\alpha_{i_0}^{-1}\alpha_iz_i=\beta
\in{\ma F}_p$. It follows that}
z_{i_0}\in k(z_1,\ldots,z_{i_0-1},
z_{i_0+1},\ldots,z_n)\quad \text{and that}\quad [K:k]\leq p^{n-1},
\end{gather*}
which is absurd. In this way we have that 
$\{\gamma_1,\ldots,\gamma_n\}\subseteq k$ is a
set linearly independent over ${\ma F}_p$.

Coming back to the expression $z=\sum_{i=1}^n\alpha_i z_i$, we have
$\wp(z)=\sum_{i=1}^n\alpha_i\gamma_i=\gamma$. In case 
$\gamma\in \wp(k)$, say $\gamma=\wp(A)$ with $A\in k$, 
we would have
$\wp(z-A)=\wp\big(\sum_{i=1}^n \alpha_iz_i-A\big)=0$. Therefore
$\sum_{i=1}^n\alpha_iz_i-A=\beta\in {\ma F}_p$. Because
$\alpha_{i_0}\neq 0$, it would follow that
\[
z_{i_0}=-\sum_{i\neq i_0}\alpha_{i_0}^{-1}\alpha_iz_i+\alpha_{i_0}^{-1}
\beta+\alpha_{i_0}A\in k(z_1,\ldots,z_{i_0-1},z_{i_0+1},\ldots,z_n),
\]
which is absurd. Whence $\gamma\notin \wp(k)$ and $[k(z):k]=p$.

With this procedure we obtain $p^n-1$ extensions of degree $p$.
Now, if $k(z)=k(w)$ with $z=\sum_{i=1}^n\alpha_iz_i$ and
$w=\sum_{i=1}^n\beta_iz_i$, $\alpha_i,\beta_i\in{\ma F}_p$, it
follows that $z=jw+c$ (see Proposition \ref{P7.0})
with $j\in{\ma F}_p^{\ast}$ and $c\in k$. Therefore
$c=0$ and $z=jw$. Hence we obtain $\frac{p^n-1}{p-1}$
different extensions of degree $p$, hence all of them.
In brief, we have obtained

\begin{proposicion}\label{P5.0} If
$K=k(z_1,\ldots,z_n)/k$ is an elementary abelian $p$--extension
of degree $p^n$ and $[k(z_i):k]=p$, $1\leq i\leq n$, then
all the subextensions of degree $p$ over $k$ are given by
$k(z)$ where $z=\alpha_1 z_1+\cdots+\alpha_n z_n$ with $\alpha_1,
\ldots,\alpha_n\in{\ma F}_p$ not all zero. \fin
\end{proposicion}

Now consider $k/k_0$ a function field with $k_0$ a perfect field,
$f(X)\in k_0[X]$ and $\G_f\subseteq k_0$. Let
$f(X)$ be given by (\ref{E2}). Let $K=k(y)$ with $f(y)=u\in k$. Let $\P$
be a place of  $k$. We have the same result as
in Theorem \ref{T4.3}, namely:

\begin{teorema}\label{T5.1} We can choose $u\in k$ such that,
either $v_{\P}(u)\geq 0$ or $v_{\P}(u)=-\lambda p^m$ with 
$\lambda\in {\ma N}$, $\gcd(\lambda,p)=1$ and $0\leq m< n$. 
In the first case $\P$ is unramified in $K/k$ and in the second $\P$
is ramified and $p^{n-m}\mid e_{\P}$.
\end{teorema}

\begin{proof}
Later on (Theorem \ref{T5.3}) we will see how to obtain
all the degree $p$ subextensions. Once we have these
subextensions we may use them to determine the
decomposition type of $\P$. Here we present another proof
in the spirit of Hasse's for Artin--Schreier extensions \cite{Has34}.

(1) If $v_{\P}(u)\geq 0$, then $u\in \o_{\P}$. With $f(y)=u$, if we define
$h(X)=\Irr(X,y,k)$, we obtain that
$h(X)\mid f(X)-u$. Therefore $f(X)-u=h(X)l(X)$ and
$f'(X)=h'(X)l(X)+h(X)l'(X)$. It follows that $f'(y)=h'(y)l(y)+0$.
In this way we obtain $h'(y)\mid f'(y)$.

Now, if $y\in\o_{\eu p}$ where $\eu p$ is a prime divisor of $K$ over
$\P$, $f'(y)=a_0\neq 0$, $v_{\P}(a_0)=0$ since $a_0\in k_0^{\ast}$.
Therefore, the local different satisfies $\eu D_{\o_{\eu p}/\o_{\P}}\mid
\langle f'(y)\rangle=\{1\}$. Hence $\eu p\nmid \eu D_{K/k}$ and
$\eu p$ is unramified.

(2) If $v_{\P}(u)<0$, let $v_{\P}(u)=-\lambda p^m$. If $m<n$, $u$
satisfies the conditions of the theorem. If $m\geq n$, we set
$\lambda p^m=\lambda_1 p^n$,
$u\in k\subseteq k_{\P}$ and $\pi$ a prime element for $\P$,
$v_{\P}(\pi)=1$, $\pi\in k$. 
Write $u$ in the form
\begin{equation}\label{Eq4.1}
u=\frac{b_{-\lambda p^m}}{\pi^{\lambda p^m}}+
\frac{b_{-\lambda p^m+1}}{\pi^{\lambda p^m-1}}+\cdots+
\frac{b_{-1}}{\pi}+b_0+b_1\pi+\cdots \in k_{\P}\cong k(\P)((\pi)).
\end{equation}

There exists $c\in k(\P)$ such that $c^{p^n}=
b_{-\lambda_1 p^n}$. Let $C\in\o_{\P}$ where $k(\P)=\o_{\P}/\P$,
be such that $c=C\bmod \P\in k(\P)$. Set $z=y-C\pi^{-\lambda_1}$. 
Then $k(z)=k(y)=K$ and
\begin{align*}
f(z)&=f(y)-f(C\pi^{-\lambda_1})\\
&=u-(C^{p^n}\pi^{-\lambda_1 p^n}+a_{n-1}
C^{p^{n-1}}\pi^{-\lambda_1 p^{n-1}}+\cdots+a_1C^p\pi^{-\lambda_1 p}+
a_0C \pi^{-\lambda_1})\\
&=\frac{b_{-\lambda p^m}}{\pi^{\lambda p^m}}+
\frac{b_{-\lambda p^m+1}}{\pi^{\lambda p^m-1}}+\cdots+
\frac{b_{-1}}{\pi}+b_0+b_1\pi+\cdots \\
&\hspace{1cm}-\Big(\frac{b_{-\lambda p^m}}{\pi^{\lambda p^m}}+
\frac{d_{-\lambda p^m+1}}{\pi^{\lambda p^m-1}}+\cdots+
\frac{d_{-\lambda_1}}{\pi^{\lambda_1}}\Big)\\
&=\sum_{i\geq -\lambda p^m+1}\gamma_i\pi^i,\quad \gamma_i\in k(\P).
\end{align*}

Therefore $v_{\P}(u-f(C\pi^{-\lambda p^m}))\geq -\lambda p^m+1>-\lambda p^m$.
In brief, if $v_{\P}(u)=-\lambda p^m$ with $m\geq n$, there exists $\delta \in k$
such that if $z:=y-\delta$, then $k(z)=k(y)$ and $v_{\P}(u-f(\delta))>-\lambda p^m$.

With this procedure we can reduce to an equation
$f(y)=u$ with $v_{\P}(u)\geq 0$ or $v_{\P}(u)=
-\lambda p^m$, $\mcd(\lambda,p)=1$,
$\lambda>0$ and $0\leq m< n$. This finishes the process.

If $u$ is of this last form, we will see that for every
$\delta\in k$, $v_{\P}(u-f(\delta))
\leq v_{\P}(u)$, that is, the value of $v_{\P}(u)$ is the maximum possible
with substitutions of the type $z=y-\delta$ with $\delta\in k$.
We have $f(\delta)=\delta^{p^n}+a_{n-1}\delta^{p^{n-1}}
+\cdots+a_1\delta^p+a_0 \delta$.

In case $v_{\P}(\delta)\geq 0$, $v_{\P}(f(\delta))\geq 0$ and since
$v_{\P}(u)<0$, it follows that
\[
v_{\P}(u-f(\delta))=\min\{v_{\P}(u),v_{\P}(f(\delta))\}=v_{\P}(u).
\]

In case $v_{\P}(\delta)<0$, $v(\delta^{p^n})=p^n v_{\P}(\delta)<p^iv_{\P}(\delta)
\leq v_{\P}(a_i)+p^i v_{\P}(\delta)=v_{\P}(a_i\delta^{p^i})$. Therefore
$v_{\P}(f(\delta))=p^nv_{\P}(\delta)=v_{\P}(\delta^{p^n})$.
Now $v_{\P}(f(\delta))\equiv 0\bmod p^n$ and $v_{\P}(u)\not\equiv
0\bmod p^n$. Thus $v_{\P}(f(\delta))\neq v_{\P}(u)$ and
$v_{\P}(u-f(\delta))=\min\{v_{\P}(u),v_{\P}(f(\delta))\}\leq v_{\P}(u)$.

Let $\eu p$ be a place of $K$ above $\P$. Then $v_{\eu p}(y)<0$
and hence $v_{\eu p}(f(y))=p^nv_{\eu p}(y)=v_{\eu p}(y^{p^n})$.
Therefore $v_{\eu p}(f(y))=p^nv_{\eu p}(y)=v_{\eu p}(u)=e_{\P}
v_{\P}(u)=-e_{\P}\lambda p^m$ and it follows that $p^{n-m}
\mid e_{\P}$ since $\mcd(\lambda,p)=1$.
\end{proof}

When $u$ is written with respect to a prime divisor $\P$ as in
Theorem \ref{T5.1}, we say that $u$ is in a
{\em normal form with respect to $\P$}. A normal form
is not unique in general.

In case $k=k_0(T)$ is a rational function field and
$\G_f\subseteq k_0$ we have:

\begin{teorema}\label{T5.2}
Let $k=k_0(T)$ be a rational function field, $f(X)\in k_0[X]$
an additive polynomial given by {\rm{(\ref{E2})}} and $K=k(y)$
an elementary abelian $p$--extension where $f(y)=u\in k$ and
$F(X)=f(X)-u\in k[X]$ is irreducible of degree $p^n$. Then we can 
choose $u$ satisfying
\begin{equation}\label{Eq5.2}
u=\sum_{i=1}^r \frac{Q_i(T)}{P_i(T)^{\alpha_i}}+ R(T)
\end{equation}
where $P_1,\ldots, P_r$ are distinct monic irreducible polynomials,
$Q_1,\ldots, Q_r\in k_0[T]$ such that $\mcd(Q_i,P_i)=1$,
$\deg Q_i<\deg P_i^{\alpha_i}$, $\alpha_i=-\lambda_i p^{m_i}>0$ with
$0\leq m_i< n$ and $\mcd(\lambda_i,p)=1$ for $1\leq i\leq r$ and
$R(T)$ is a polynomial such that if $R(T)\notin k_0$ then
$\deg R(T)=\lambda_0 p^m>0$ with $\mcd(\lambda_0, p)=1$,
$0\leq m<n$ and if $R(T)\in k_0$ then either $R(T)=0$ or
$R(T)\notin f(k_0)=\{f(\delta)\mid \delta\in k_0\}$.

Furthermore $P_1,\ldots,P_r$ are precisely the finite prime divisors
of $k$ ramified in $K$ and $\p$ is ramified if and only if $R(T)
\notin k_0$.
\end{teorema}

\begin{proof}
Let $f(y)=u=\frac{g(T)}{h(T)}$ with $\mcd(g(T),h(T))=1$. Expanding in
partial fractions, we obtain
\[
u=\sum_{i=1}^r\sum_{j=1}^{\beta_i}\frac{Q_j^{(i)}(T)}{P_i(T)^j} +R(T),
\]
where $\deg Q_j^{(i)}<\deg P_i^j$ for any 
$1\leq j\leq \beta_i$, $1\leq i\leq r$ and $R(T)\in k_0[T]$.

If $\beta_1=\lambda p^n>0$, we can choose $C\in k[T]$
such that 
\[
C(T)^{p^n}\equiv Q_1^{(\beta_1)}(T)\bmod P_1(T)
\]
because $k_0[T]/(P_1)$ is a perfect field.
Using the substitution $z=y-C^{\lambda}$ we get $K=k(z)$
and $f(z)=f(y)-f(C^{\lambda})=u-f(C^{\lambda})=w$. It follows that
the valuations of $w$ for an arbitrary prime divisor $\P\neq \p$ 
of $k_0$ satisfy:
\[
v_{\P}(w) \begin{cases} \geq 0&\text{if $v_{\P}(u)\geq 0$}\\
=-\beta_j& \text{if $\P$ is the prime divisor associated to $P_j(T)$ for $2\leq j\leq r$}\\
>-\beta_1&\text{if $\P$ is the prime divisor associated to $P_1(T)$}.
\end{cases}
\]
Repeating this process, we  obtain, for $\beta_j$, $2\leq j\leq r$,
that $K=k(y)$ with $f(y)=u$ and $u$ is in the form 
(\ref{Eq5.2}) except possibly for $R(T)$.

Now if $R(T)=b_dT^d+\cdots +d_0$ satisfies $d=\lambda p^n$
we make the substitution $y=z-cT^{\lambda}$ where $c^{p^n}=b_d$. 
Keeping on this process we finally obtain that either
$R(T)\in k_0$ or $\deg R(T)= \lambda p^m$ with $0\leq m< n$. 
Finally, if $R(T)\in k_0$ and $R(T)=f(\delta)$ for some 
$\delta\in k_0$, we take $z=y-\delta$.

The type of ramification is an immediate consequence
of Theorem \ref{T5.1}.
\end{proof}

\begin{definicion}\label{D5.2'}{\rm{When the 
equation $f(y)=u$ defining the extension
$K = k(y)$ satisfies the conditions of
Theorem \ref{T5.2}, we say that the equation
is in a {\em reduced form}.
}}
\end{definicion}

Note that the reduced form is not unique in general.

Next, we present the results mentioned by Garcia and 
Stichtenoth \cite{GarSti91} on additive polynomials whose
roots belong to the base field.

Let $K=k(y)$ with $f(y)=u\in k$, $f(X)\in k[X]$ a monic separable
additive polynomial whose roots are in $k$ and $F(X)=f(X)-u
\in k[X]$ is irreducible, $G=\Gal(K/k)\cong \G_f\cong C_p^n$.

\begin{teorema}\label{T5.3} The subextensions of degree $p$
over $k$ of $K/k$ are given by $k(z_{\H})$ with
\[
z_{\H}^p-z_{\H}=\frac{u}{f_{\H}(\varepsilon_{\H})^p},
\]
where $\H<\G_f$ is a subspace of $\G_f$ of codimension
$1$, $f_{\H}(X)=\prod_{\delta\in\H}(X-\delta)$ and $\G_f=
\H+{\ma F}_p\varepsilon_{\H}$, $\varepsilon_{\H}\in \G_f$. Furthermore,
$k(z_{\H})$ is the fixed field under $\H$: $k(z_{\H})=K^{\H}$
and $\Gal(k(z_{\H})/k)=\langle \sigma_{\varepsilon_{\H}}\rangle$ where
$\sigma_{\varepsilon_{\H}}(y)=y+\varepsilon_{\H}$.
We have $\sigma_{\varepsilon_{\H}}(z_{\H})=z_{\H}+1$.
\end{teorema}

\begin{proof} 
We have $f(X)=\prod_{\alpha\in{\ma F}_p}
f_{\H}(X-\alpha \varepsilon_{\H})$. Now, $f_{\H}(X)$ is
an additive polynomial (see for instance \cite[proof of
Proposition 13.4.10]{Vil2006}) and $f_{\H}(X-\alpha\varepsilon_{\H})=
f_{\H}(X)-\alpha f_{\H}(\varepsilon_{\H})$.

Denote $Y:=f_{\H}(X)$. Then
\begin{align*}
f(X)&=\prod_{\alpha=0}^{p-1}(Y-\alpha f_{\H}(\varepsilon_{\H}))
=f_{\H}(\varepsilon_{\H})^p\cdot\prod_{\alpha=0}^{p-1}\Big(\frac{Y}{
f_{\H}(\varepsilon_{\H})}-\alpha\Big)\\
&=f_{\H}(\varepsilon_{\H})^p
\Big(\Big(\frac{Y}{f_{\H}(\varepsilon_{\H})}\Big)^p-\frac{Y}{f_{\H}(
\varepsilon_{\H})}\Big).
\end{align*}

So, 
\begin{gather*}
f(X)=f_{\H}(X)^p-f_{\H}(\varepsilon_{\H})^{p-1}f_{\H}(X)=
f_{\H}(\varepsilon_{\H})^p\Big(\Big(\frac{f_{\H}(X)}{f_{\H}(\varepsilon_{\H})}
\Big)^p-\Big(\frac{f_{\H}(X)}{f_{\H}(\varepsilon_{\H})}\Big)\Big).
\intertext{That is,}
f(X)=f_{\H}(\varepsilon_{\H})^p\wp\Big(\frac{f_{\H}(X)}{f_{\H}(\varepsilon_{
\H})}\Big)=\wp_{f_{\H}(\varepsilon_{\H})}(f_{\H}(X)).
\end{gather*}

In this way we obtain $f(X)=f_{\H}(\varepsilon_{\H})^p(z^p-z)$ 
where $z=\frac{Y}{f_{\H}(\varepsilon_{\H})}=\frac{f_{\H}(X)}{
f_{\H}(\varepsilon_{\H})}$. Let 
\begin{equation}\label{Eq5.3}
z_{\H}:=\frac{f_{\H}(y)}{f_{\H}(
\varepsilon_{\H})}.
\end{equation}
Then 
\[
z_{\H}^p-z_{\H} =\frac{f(y)}{f_{\H}(\varepsilon_{\H})^p}=
\frac{u}{f_{\H}(\varepsilon_{\H})^p} \qquad\text{or}\qquad
\wp_{f_{\H}(\varepsilon_{\H})}\big(
f_{\H}(\varepsilon_{\H})z_{\H}\big)=u.
\]

Furthermore, if the set 
$\{\varepsilon_1,\ldots,\varepsilon_{n-1}\}$ is an
$\ma F_p$--basis of $\H$ and $\varepsilon_n:=\varepsilon_{\H}$,
then $\{\varepsilon_1,\ldots,\varepsilon_{n-1},\varepsilon_n\}$
is a basis of $\G_f$ over $\ma F_p$. We have that if $\sigma_{
\varepsilon_i}(y)=y+\varepsilon_i$, then
$G=\langle
\sigma_{\varepsilon_1},\ldots,\sigma_{\varepsilon_{n-1}},
\sigma_{\varepsilon_n}\rangle$ and we have, for $1\leq i\leq n-1$,
\begin{gather*}
\sigma_{\varepsilon_i}(z_{\H})=\sigma_{\varepsilon_i}\Big(
\frac{f_{\H}(y)}{f_{\H}(\varepsilon_n)}\Big)=\frac{f_{\H}(y+
\varepsilon_i)}{f_{\H}(\varepsilon_n)}=\frac{f_{\H}(y)+
f_{\H}(\varepsilon_i)}{f_{\H}(\varepsilon_n)}=\frac{
f_{\H}(y)+0}{f_{\H}(\varepsilon_n)}=z_{\H},\\
\intertext{and}
\sigma_{\varepsilon_n}(z_{\H})=\frac{f_{\H}(y+\varepsilon_n)}{
f_{\H}(\varepsilon_n)}=z_{\H}+1,
\end{gather*}
so that $k(z_{\H})/k$ is a subextension of $K/k$ of degree $p$, the
field $k(z_{\H})$ is the fixed field under $\H$ and 
$\Gal(k(z_{\H})/k)\cong\langle \sigma_{
\varepsilon_{\H}}\rangle$ where $\sigma_{\varepsilon_{\H}}(y)=
y+\varepsilon_{\H}$.
\end{proof}

\begin{observacion}\label{O5.3'} {\rm{
Theorem \ref{T5.3} must be compared with
Theorem \ref{T7.2} which is more general but less explicit.
}}
\end{observacion}

\begin{teorema}\label{T5.4} Let $f(X)\in k[X]$ be a monic separable
additive polynomial of degree $p^n$ with $\G_f\subseteq k$. 
Let $K/k$ be an elementary abelian $p$--extension of degree
$p^n$. Then there exist $y\in K$ and $u\in k$ such that
$K=k(y)$ and $f(y)=u$.
\end{teorema}

\begin{proof}
Let $y_1,\ldots, y_n\in K$ be such that $K=k(y_1,\ldots, y_n)$ and
$y_i^p-y_i=\gamma_i\in k$. Let $G=\Gal(K/k)$, $G=\langle
\sigma_1,\ldots, \sigma_n\rangle$ with $\sigma_i(y_j)=
y_j+\delta_{ij}$ where $\delta_{ij}$ is the Kronecker delta. Let
$\{\mu_1,\ldots,\mu_n\}$ be a basis of $\G_f$ over $\ma F_p$.
Let $y:=\sum_{i=1}^n\mu_i y_i$ and let $f(X)$ be given as in (\ref{E2}).

We have $y_i^p=y_i+\gamma_i$, $y_i^{p^2}=y_i^p+\gamma_i^p=
y_i+\gamma_i+\gamma_i^p$ and in general $y_i^{p^m}=y_i+
l_m(\gamma_i)$ where $l_m(\gamma_i)=\gamma_i+\gamma_i^p
+\cdots+\gamma_i^{p^{m-1}}$, $m\in \ma N$. Then
\[
f(\mu_iy_i)=\sum_{j=0}^n a_j(\mu_i y_i)^{p^j}=
\sum_{j=0}^n a_j(\mu_i^{p^j}y_i+\mu_i^{p^j} l_j(\gamma_i))=
y_i f(\mu_i)+h_i,
\]
with $h_i=\sum_{j=0}^n a_j\mu_i^{p^j}l_j(\gamma_i)$ and
$y_if(\mu_i)=0$ because $\mu_i\in \G_f$. Therefore,
\[
f(y)=f\big(\sum_{i=1}^n\mu_iy_i\big)=\sum_{i=1}^nf(\mu_iy_i)=
\sum_{i=1}^nh_i=u\in k.
\]

If $\sigma\in G$, there exist $\nu_1,\ldots,\nu_n\in \ma F_p$ such that
$\sigma=\sigma_1^{\nu_1}\cdots\sigma_n^{\nu_n}$ and
\[
\sigma(y)=\sigma\big(\sum_{i=1}^n \mu_iy_i\big)=\sum_{i=1}^n
\mu_i(y_i+\nu_i)=y+\sum_{i=1}^n\nu_i\mu_i.
\]
Thus $\sigma=\Id\iff \sigma(y)=y\iff \sum_{i=1}^n \nu_i\mu_i=0
\iff \nu_1=\cdots=\nu_n=0$, confirming that $\{\mu_1,\ldots,\mu_n\}$
is a basis of $\G_f$ over 
${\ma F}_p$ and that $G=\Gal(K/k)=\langle \sigma_1,
\ldots,\sigma_n\rangle$. 
\end{proof}

\section{Decomposition of prime divisors in elementary abelian
$p$--extensions of function fields}\label{S6}

Let us to consider $k=k_0(T)$ a rational function field
where we assume that $k_0$ is a finite field with
$\G_f\subseteq k_0$. The aim of this section is to describe the
decomposition of the unramified primes in an elementary
abelian $p$--extension $K$ of $k$ given by $K=k(y)$ where
$f(y)=u$, $f(X)\in k_0[X]$ given by (\ref{E2}) 
and such that $\G_f\subseteq k_0$.
Note that the decomposition group of any unramified prime
is a cyclic group and therefore it is of order $1$ or $p$. We will
assume that the extension $K/k$ is geometric.

A fundamental result that we will use in this section is
the decomposition of $\p$ in Artin--Schreier extensions.

\begin{proposicion}\label{P6.1}
Let $L/K$ be a cyclic extension of degree $p$ such that
$K=k(w)$ with $w\in L$ given in the form
\begin{gather}\label{Eq6.1}
w^p-w=u=\sum_{i=1}^r\frac{Q_i}{P_i^{e_i}} + f(T)=
\frac{Q}{P_1^{e_1}\cdots P_r^{e_r}}+f(T),
\end{gather}
where $P_i\in R_T^+$, $Q_i\in R_T$, 
$\mcd(P_i,Q_i)=1$, $e_i>0$, $p\nmid e_i$, $\deg Q_i<
\deg P_i^{e_i}$, $1\leq i\leq r$,
$\deg Q<\sum_{i=1}^r\deg P_i^{e_i}$, $f(T)\in R_T$,
with $p\nmid \deg f$ in case $f(T)\not\in k_0$
and $f(T)\notin \wp(k_0)$ when $f(T)
\in k_0^{\ast}$. Then the prime divisor $\p$ is
\l
\item decomposed if $f(T)=0$.
\item inert if $f(T)\in k_0$ and $f(T)\not\in \wp(k_0)$.
\item ramified if $f(T)\not\in k_0$ (so that $p\nmid\deg f$). \fin
\end{list}
\end{proposicion}

The following example illustrates several of the results
obtained.

\begin{ejemplo}\label{Ej5.2(1)}{\rm{
Let $k=\mathbb{F}_{27}(T)$ and $K=k(z)/k$ be the $3$--elementary
abelian extension of degree $27$, defined by the equation:
\begin{gather*}
z^{27}-z=\frac{1}{(T+1)^{54}}+\frac{1}{T+1}+T^{9}+T^{3}+T+\omega+1=u(T),
\end{gather*}
where $\omega\in \mathbb{F}_{27} $, $\omega^{3}=
\omega+2$ and $\mathbb{F}_{27}=\mathbb{F}_{3}(\omega)$. 
First, note that if $y=z-\frac{1}{\left(T+1\right)^{2}}$, 
then $k(y)=k(z)$ and $y^{27}-y=(z-\frac{1}{\left(T+
1\right)^{2}})^{27}-(z-\frac{1}{\left(T+1\right)^{2}})=u(T)-
\frac{1}{\left(T+1\right)^{54}}+\frac{1}{\left(T+1\right)^{2}}$. 
That is
\[
y^{27}-y=\frac{1}{\left(T+1\right)^{2}}+\frac{1}{T+1}
+T^{9}+T^{3}+T+\omega+1=r(T).
\]
From Theorem \ref{T4.3}, the prime divisor
$\mathcal{P}_{1}$, associated to the irreducible polynomial
$p_{1}(T)=T+1$ and $\mathcal{P}_{\infty}$, are ramified.
Next, we compute the ramification index 
for these prime divisors.
We can use Proposition \ref{P3.1} to obtain all
the $13=\frac{3^{3}-1}{3-1}$ 
Artin-Schreier subextensions. Three of these 
extensions that generate the
whole extension are given by:  
\begin{align*}
y_{1}^{3}-y_{1}&=r(T),\\
y_{1}&=y^{9}+y^{3}+y.\\
y_{2}^{3}-y_{2}&=\omega r(T),\\
y_{2}&=(\omega y)^{9}+(\omega y)^{3}+\omega y.\\
y_{3}^{3}-y_{3}&=\omega^{2} r(T),\\
y_{3}&=(\omega^{2} y)^{9}+(\omega^{2} y)^{3}+\omega^{2} y.
\end{align*}

The following diagram represents these 3--subextensions and
some elementary abelian 3--extensions of degree 9.
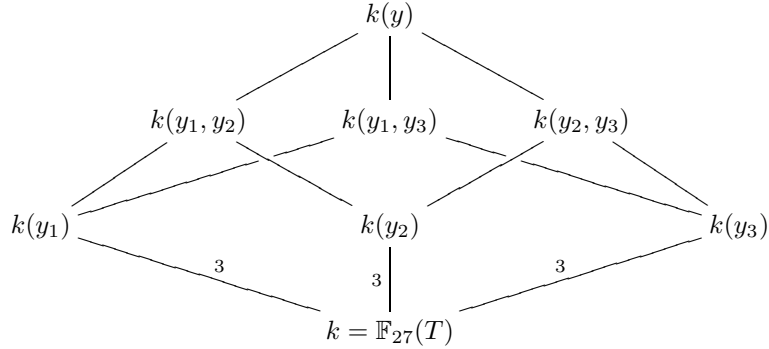
\[
\xymatrix{ 
& & k(y)   & & \\
&k(y_{1},y_{2}) \ar@{-}[ur]   &   k(y_{1},y_{3})\ar@{-}[u] & 
 k(y_{2},y_{3})\ar@{-}[ul]  &\\
k(y_{1}) \ar@{-}[ur] \ar@{-}[urr]|!{[ur];[rr]}\hole & &  
 k(y_{2})\ar@{-}[ul] \ar@{-}[ur] &  & 
  k(y_{3})\ar@{-}[ul] \ar@{-}[ull] |!{[ll];[ul]}\hole\\
 &  & k=\mathbb{F}_{27}(T) \ar@{-}[urr]^{3} 
 \ar@{-}[u]^{3} \ar@{-}[ull]_{3} &  & 
}
\]

Note that $\mathcal{P}_{1}$ 
is fully ramified, that is
$e(\wp_{1}|\mathcal{P}_{1})=27$ for a place
$\wp_{1}$ above $\mathcal{P}_{1}$, since
$v_{\mathcal{P}_{1}}(r(T))=-2$ is relatively prime to
$3$. On the other hand, if $z_{1,1}=y_{1}-T^{3}$, 
then $k(z_{1,1})=k(y_{1})$ and $z^{3}_{1,1}-z_{1,1}=
\frac{1}{\left(T+1\right)^{2}}+\frac{1}{T+1}+2T^{3}+T+
\omega+1$. Now, let $z_{1,2}=z_{1,1}+T-\omega^{2}$.
We have $k(z_{1,2})=k(z_{1,1})=k(y_{1})$ and
\[
z^{3}_{1,2}-z_{1,2}=\frac{1}{\left(T+1\right)^{2}}
+\frac{1}{T+1}=r_{1}(T).
\]
From Proposition \ref{P6.1}, 
$\mathcal{P}_{\infty}$ decomposes in $k(y_{1})/k$.

In the extension $k(y_{2})/k$, we let
$z_{2,1}=y_{2}+(2\omega+2)T^{3}$.
Then $k(z_{2,1})=k(y_{2})$
and $z^{3}_{2,1}-z_{2,1}=(2\omega+1)T^{3}+\omega T+
\omega^{2}+\omega+\frac{\omega}{\left(T+1\right)^{2}}+
\frac{\omega}{T+1}$. Again, if we make the substitution:
$z_{2,2}=z_{2,1}-2\omega T$, in this extension, we obtain
$k(z_{2,2})=k(z_{2,1})=k(y_{2})$ and
\[
z^{3}_{2,2}-z_{2,2}= \frac{\omega}{\left(T+1\right)^{2}}
+\frac{\omega}{T+1}+ \omega^{2}+\omega=r_{2}(T),
\]
because $\omega^{2}+\omega\notin
\wp(\mathbb{F}_{27})$, from Proposition
\ref{P6.1} it follows that the infinite prime is inert in
$k(y_{2})/k$. Finally, with the substitutions:
 $z_{3,1}=y_{3}-(2\omega^{2}+\omega+2)T^{3}$ 
and $z_{3,2}=z_{3,1}-(2\omega^{2}+2)T$, we obtain $k(z_{3,2})
=k(z_{3,1})=k(y_{3})$, where
\[
z^{3}_{3,2}-z_{3,2}=
\frac{\omega^{2}}{\left(T+1\right)^{2}}+\frac{\omega^{2}}{T+1}
+2T+\omega^{2}+\omega+2=r_{3}(T).
\]
Therefore, 
$\mathcal{P}_{\infty}$ is ramified in $k(y_{3})/k$.
Hence, the decomposition field of the infinite prime $\p$
is $k(y_{2})$ and the inertia field is
$k(y_{2},y_{3})$. We obtain 
$\p$ has ramification index equal to $3$,
inertia degree equal to $3$ and has $3$ prime
divisors in $K$ above it.
}}
\end{ejemplo}

\bigskip

Returning to our study, 
note that with the hypothesis of being of degree $p$, separable and
with the roots in the base field, essentially there is a unique additive
polynomial of degree $p$. If $X^p+aX$ is an additive
polynomial, their roots are
equal to $i\alpha$ with $0\leq i\leq p-1$ and
$\alpha=\sqrt[p-1]{-a}$ a fixed nonzero root of $X^p+aX$.
We are assuming that $\alpha\in k_0^{\ast}$. Then $\alpha^p=-\alpha a$ and
\[
X^p+aX=\alpha^p\big(Z^p-Z\big)\quad\text{with}\quad Z=\frac{X}{\alpha},
\quad \alpha\in \f,
\]
that is, $X^p+aX=\alpha^p\wp\big(\frac{X}{\alpha}\big)$ with 
$\alpha^{p-1}=-a$.

We come back to the extension $K=k(y)$, $\Gal(K/k)\cong \G_f$. 
We have that the decomposition of $\p$ in $K$, in case of being
unramified is

\begin{proposicion}\label{P6.2}
Let $K=k(y)/k$ with $f(y)=u$ and $u$ given in a reduced form 
{\rm{(\ref{Eq5.2})}}. We assume that $\p$ is unramified in $K/k$.
If $R(T)=0$, then $\p$ decomposes fully in $K/k$.
Conversely, if $\p$ decomposes fully, then there exists a reduced
form  $f(y)=u$ where $R(T)=0$.
\end{proposicion}

\begin{proof}
If $R(T)=0$ then, from Theorem \ref{T5.3} and Proposition
\ref{P6.1}, it follows that $\p$ decomposes in every subextension
of degree $p$ and therefore $\p$ is fully decomposed in $K/k$.

Conversely, assume that $\p$ is fully decomposed. Let
$K=k(y_1,\ldots,y_n)$ with $k(y_i)/k$ cyclic extensions of degree
$p$ given by Artin--Schreier equations in reduced form.
Since $\p$ decomposes in all of them, from Theorem \ref{T5.4},
it follows $K=k(y_0)$ with $y_0=\sum_{
i=1}^n\mu_iy_i$, $\{\mu_1,\ldots,\mu_n\}$ a basis of $\G_f$ over
${\ma F}_p$, and in the reduced form of $f(y_0)=u_0$,
the polynomial corresponding to the behavior of $\p$ is $0$. 
\end{proof}

In the special case of the additive polynomial $f(X)=X^q-X$, 
we are going to prove more than Proposition \ref{P6.2}.
To this end, we prove

\begin{lema}\label{L6.2'} Let $S\in{\ma F}_{q^m}=k_0$. 
Then we have
$\mu S\in \im\wp$ for all $\mu\in\F$ if and only if there exists
$\lambda\in {\ma F}_{q^m}$ such that $S=\lambda^q-\lambda$.
\end{lema}

\begin{proof} We consider the homomorphism
$g\colon{\ma F}_{q^m}\lra {\ma F}_{q^m}$,
$g(\lambda)=\lambda^q-\lambda$. Then $\ker g=\F$ and $|\im g|=
\frac{q^m}{q}$. The Artin--Schreier homomorphism,
$\wp\colon {\ma F}_{q^m}\to
{\ma F}_{q^m}$, $\wp(\lambda)=\lambda^p-\lambda$, satisfies
$\ker \wp={\ma F}_{p}$ and $|\im\wp |=\frac{q^m}{p}$. Finally we consider
$h\colon {\ma F}_{q^m}\lra {\ma F}_{q^m}$ given by $h(\lambda)=\lambda
+\lambda^p+\cdots+\lambda^{p^{n-1}}$, $q=p^n$. Then 
$h(\lambda)^p-h(\lambda)=\lambda^q-\lambda$.

We have that 
\[
(\wp\circ h)=(h\circ \wp)=g.
\] 
Now, if $S=\lambda^q-\lambda$
for some $\lambda\in {\ma F}_{q^m}$, then
$\mu S=(\mu\lambda )^q-(\mu \lambda)$ for all $\mu\in \F$
and we have 
\[
(\mu\lambda )^q-(\mu \lambda)=g(\mu\lambda)=\wp(h(\mu\lambda))
\in \im \wp,
\]
for all $\mu\in\F$.

Conversely, assume that $A:=\big\{\mu S\big\}_{\mu\in\F}
\subseteq \im\wp$. If $S=0$ there is nothing to prove. We consider
$S\neq 0$. 
In case that $S\notin\im g$ we would have $\mu S\notin \im g$ 
for all $\mu\in\f$ since otherwise, if $\mu S=\lambda^q-\lambda$ for some
$\lambda\in {\ma F}_{q^m}$ then for all $\mu\in\F$, $\mu^q=\mu$ and
\[
S=\frac{\lambda^q}{\mu}-\frac{\lambda}{\mu}
=\Big(\frac{\lambda}{\mu}\Big)^q
-\Big(\frac{\lambda}{\mu}\Big)\in \im g,
\]
which is absurd.

Under this assumption
the additive subgroup $A=\big\{\mu S\big\}_{\mu\in \F}$ of
${\ma F}_{q^m}$, which is of cardinality $q$, satisfies
$A\cap \im g=\{0\}$. Since $A\subseteq \im \wp$ and because
$g=\wp\circ h$, that is, $\im g\subseteq \im \wp$, it follows that
$A+\im g\subseteq \im\wp$. But on the other hand
\[
|A+\im g|=|A||\im g|=q\cdot \frac{q^m}{q}=q^m>\frac{q^m}{p}=|\im \wp|,
\]
which is a contradiction.
Thus $A\cap \im g\neq \{0\}$. Therefore there exists $\mu\in\f$
with $\mu S\in \im g$ which implies that $S\in\im g$.
This finishes the proof.
\end{proof}

Let $f(X)$ be an additive polynomial given by (\ref{E2}). 
The result of Lemma \ref{L6.2'} should hold in this case
but we don't have a proof. Let
$\G_f=\L\{\varepsilon_1,\ldots,\varepsilon_n\}$ and consider the
following $n$ subspaces of $\G_f$ of codimension $1$: 
\[
\H_i:=
\L\{\varepsilon_1,\ldots,\varepsilon_{i-1},
\varepsilon_{i+1},\varepsilon_n\},\quad 1\leq i\leq n.
\]
Let $f_{\H_i}(X)=f_i(X)$, $a_i=f_i(\varepsilon_i)\neq 0$. 

What we need is:
\begin{gather}\label{Eq*}
\bigcap_{i=1}^n\im\wp_{a_i}=\im f.
\end{gather}

We have:

\begin{proposicion}\label{P6.2''} Let $K=k(y)/k$ 
with $f(y)=u$ and $u$ given in a reduced form
{\rm{(\ref{Eq5.2})}}. We assume that
$\p$ is unramified in $K/k$. If $f(X)=X^q-X$ 
or if {\rm{(\ref{Eq*})}} holds, then $\p$ decomposes fully if
and only if $R(T)=0$.
\end{proposicion}

\begin{proof}
First we consider the particular case $f(X)=X^q-X$.
We have that $\p$ decomposes fully in $K/k$ if and only
if it decomposes fully in every intermediate extension of
degree $p$: $y_{\mu}^p-y_{\mu}=\mu u$ where
$y_{\mu}=h(\mu y)=(\mu y)^{p^{n-1}}+\cdots +(\mu y)^p + 
(\mu y)$, $\mu\in \f$. The latter is equivalent to
$\mu R(T)\in\im\wp({\ma F}_{q^m})$ for all $\mu\in\F$.
From Lemma \ref{L6.2'} the latter is equivalent to
$R(T)=\lambda^q-\lambda$ for some 
$\lambda\in{\ma F}_{q^m}$. Since $u$ is in reduced form,
it follows that $R(T)=0$.

For the general case, from Theorem \ref{T5.3}, $\p$ 
decomposes fully in $K/k$ if and only if $\p$ decomposes in
every subextension of degree $p$ given by
\begin{equation}\label{infinito}
z_{\H}^p-z_{\H}=\frac{u}{f_{\H}(\varepsilon_{\H})^p},
\end{equation}
for every hyperplane $\H$ of $\G_f$. The term
(\ref{infinito}) that determines the behavior of $\p$
is $\frac{R(T)}{f_{\H}(\varepsilon_{\H})^p}$. Therefore
$\p$ decomposes fully in (\ref{infinito}) if and only if
$R(T)\in\im \wp_{f_{\H}(\varepsilon_{\H})}$. Now the result
is a consequence of equation (\ref{Eq*}).
\end{proof}

Next, we present a result similar to that of
Proposition \ref{P6.1} for a certain family of elementary
abelian $p$--extensions. Note that when the
infinite prime is ramified, we do not give the
precise behavior of the prime in the extension.
 
\begin{corolario}\label{C6.2'''}
Let $K=k_0(T)$ be a rational function field and
$K = k(y)$ be an elementary abelian $p$--extension
of $k$ given by a reduced form
\begin{gather*}
y^q-y=u=\sum_{i=1}^r\frac{Q_i(T)}{P_i(T)^{\alpha_i}} + R(T)
\end{gather*}
satisfying the conditions of Theorem {\rm{\ref{T5.2}}}.

Then the prime divisor $\p$ is
\l
\item fully decomposed if $R(T)=0$.
\item not ramified with inertia degree $p$ if
$R(T)\in k_0$ and $R(T)\not\in\{\lambda^q - \lambda
 | \lambda \in k_0 \}$.
\item ramified if $R(T)\not\in k_0$. \fin
\end{list}
\end{corolario}

Let $K=k(y)$ be given by (\ref{Eq5.2}) and let $\P_i$ be the prime
divisors associated to $P_i(T)$, $1\leq i\leq r$. Let $\P$
be a prime divisor such that $\P\notin\{\P_1,\ldots,\P_r,\p\}$. 
Then $\P$ is either fully decomposed in $K/k$ or has inertia
degree equal to $p$. The following results establish the
decomposition type of $\P$.

First we study the Artin--Schreier case.

\begin{proposicion}\label{P6.3} Let $\P$ and $K/k$ be as before. 
Let us write equation {\rm{(\ref{Eq6.1})}} as
\[
y^p-y=u(T),
\]
with $u(T)=\frac{g(T)}{h(T)}\in k=k_0(T)$ such that
$\gcd(g(T),h(T))=1$. Let
$P(T)\in R_T^+$ be the irreducible polynomial associated to
$\P$, say $\deg P(T)=m$. Let $\nu\in k'$ be a root of $P(T)$, where
$k'$ is the decomposition field of the polynomial $P(T)$.
Then $\P$ decomposes fully in $K/k$ if and only
 if $u(\nu)\in \wp(k')$.
\end{proposicion}

\begin{proof} We have $[k':k_0]=m$. Let
$\nu=\nu_1,\ldots,\nu_m$ be the roots of $P$ in $k'$, $P(T)=\prod_{i=1}^m
(T-\nu_i)$. We have that $\P$ decomposes fully in $k'(T):=k_m/k=
k_0(T)$. Here $k_m$ denotes the constant extension of $k$ of degree
$m$. Since we are assuming that the extension $K/k$ is geometric, it
follows that $k_m\cap K=k$ and therefore $\P$ decomposes fully in $K/k$
if and only if $\Q$ decomposes fully in $K_m/k_m$ where $K_m=
Kk_m$ and $\Q$ is a prime in $k_m$ above $\P$.
\[
\xymatrix{K\ar@{-}[rr]\ar@{-}[d]&&K_m\ar@{-}[d]\\
k\ar@{-}[rr]_{\substack{\text{$\P$ decomposes}
\\ \text{fully}}}&&k_m}
\]

Say that $\Q$ is the prime divisor associated to $T-\nu\in k_m$. 
Since $v_{\P}(u(T))\geq 0$ it follows that $v_{\Q}(u(T))\geq0$ so that
$T-\nu\nmid h(T)$ and $h(\nu)\neq 0$. Furthermore $g(\nu)=0\iff v_{\P}(u(T))>0$.
We have $\deg_{k_m} \Q=1$. We make $\Q$ the infinite prime in
$k_m$, that is, let $T'=\frac{1}{T-\nu}$,
$(T')_{k_m}=\frac{\Q'_0}{\Q'_{\infty}}=\frac{\Q_{\infty}}{\Q}$, where $(T)_{k_m}=\frac{
\Q_0}{\Q_{\infty}}$. We have $T=\frac{1}{T'}+\nu$.

Write $u_1(T'):=u(T)=u\big(\frac{1}{T'}+\mu\big)=\frac{g\big(\frac{1}{T'}+\nu\big)}
{h\big(\frac{1}{T'}+\nu\big)}=\frac{g\big(\frac{1}{T'}(1+T'\nu)\big)}
{h\big(\frac{1}{T'}(1+T'\nu)\big)}$.

Let $g(T)=a_sT^s+a_{s-1}T^{s-1}+\cdots+a_1T+a_0$, $a_s\neq 0$, $a_i\in \F$, 
$0\leq i\leq s$; $h(T)=b_tT^t+b_{t-1}T^{t-1}+\cdots+b_1T+b_0$, $b_t\neq 0$, $b_j\in \F$, 
$0\leq j\leq t$. Then
\begin{gather*}
g\Big(\frac{1}{T'}(1+T'\nu)\Big)=\frac{1}{(T')^s}\Big(a_s+\cdots+g(\nu)(T')^s\Big)=
\frac{1}{(T')^s}g_1(T');\\
h\Big(\frac{1}{T'}(1+T'\nu)\Big)=\frac{1}{(T')^t}\Big(b_t+\cdots+h(\nu)(T')^t\Big)=
\frac{1}{(T')^t}h_1(T').
\intertext{It follows that}
\deg_{T'}g_1(T')\leq s;\quad\deg_{T'}h_1(T')=t.
\intertext{Therefore}
\deg_{T'}u_1(T')=\deg_{T'}g_1(T')-s\leq 0,\quad\text{and}\quad v_{\Q}(u_1(T'))=s-\deg_{
T'}(g_1(T'))\geq 0.
\end{gather*}

The reduced form of $u_1(T')$ is
\[
u_1(T')=\sum_{j=1}^{r'}\frac{Q'_j(T')}{(P'_j(T'))^{e'_j}}+u(\nu).
\]
Thus $\Q$ decomposes fully in $K_m/k_m\iff u(\nu)\in 
\wp(k')$. This proves the proposition.
\end{proof}

The general case is consequence of Proposition \ref{P6.3}. Indeed,
let $K=k(y)$ where $y$ is given in a reduced form (\ref{Eq5.2}). We
assume that $\P$ is not ramified in $K/k$. From Theorem \ref{T5.3}, 
using the notation given there, we have that every subextension of
$K$ of degree $p$ over $k$ is given by (\ref{infinito}).

As a consequence of these expressions and from Propositions
\ref{P6.2''} and \ref{P6.3}, we obtain the main result of this section.

\begin{teorema}\label{T6.4} With the above notations, let $\P$ be
a non--ramified in $K/k$ prime divisor of degree $m$. Then we have
that $\P$ decomposes fully in $K/k\iff$ for every hyperplane $\H$ holds
\begin{gather*}
u(\nu)\frac{1}{f_{\H}(\varepsilon_{\H})^p}\in \wp (k'),
\intertext{and $\P$ has inertia degree $p$ in $K/k \iff$ 
there exists an hyperplane $\H$ such that}
u(\nu)\frac{1}{f_{\H}(\varepsilon_{\H})^p}\notin \wp (k').
\end{gather*}

Equivalently, if {\rm{(\ref{Eq*})}} holds, then $\P$ decomposes
fully in $K/k \iff u(\nu)\in f(k')$. \fin
\end{teorema}

\section{Generators of elementary abelian $p$--extensions}\label{S7}

Let $k$ be an arbitrary field of characteristic $p>0$ and let
$f(X)\in k[X]$ be an additive polynomial given by (\ref{E2}) such that
$\G_f\subseteq k$. Let $u\in k$ be such that $F(X)=f(X)-u\in k[X]$ 
is irreducible. Let $K=k(y)$ where $f(y)=u$,  $\Gal(K/k)
\cong \G_f$. We have that $K/k$ has $\frac{(p^n-1)(p^n-p)\cdots (p^n-p^{m-1})}
{(p^m-1)(p^m-p)\cdots (p^m-p^{m-1})}$ subextensions $k\subseteq E
\subseteq K$ such that $[E:k]=p^m$.

We want to study the relation between $y$ and $z$ in case
$K=k(y)=k(z)$ and $f(z)=\chi\in k$.
The next result is well known.

\begin{proposicion}\label{P7.0}
If $k(y_1)$ y $k(y_2)$ are cyclic extensions of degree $p$ of $k$
with $\wp (y_i)=x_i\in k$,
$i=1,2$, the following two statements are equivalent:
\l
\item  $k(y_1) =k(y_2)$,
\item there exist $j\in {\ma F}_p^*$ and $z \in k$
such that $y_1= j y_2 + z$ and
 $x_1= j x_2 + \wp (z)$. \hfill{\fin}
\end{list}
\end{proposicion}

Next theorem generalizes Propostition \ref{P7.0}. 
This is the main result of this section.

\begin{teorema}\label{T7.1} With the above notation we have that
the following statements are equivalent:
\l
\item $k(y)=k(z)$,

\item there exist $A_{n-1},A_{n-2},\ldots, A_1,A_0
\in \G_f$ satisfying 
\begin{gather*}
\elemental A{\beta}n=0\\
\text{with}\quad \beta \in  \G_f \iff \beta=0
\end{gather*}
and $D\in k$ such that
\begin{equation}\label{Eq7.1}
z=\elemental Ayn+D.
\end{equation}
\end{list}
\end{teorema}

Theorem \ref{T7.1} is an immediate consequence of the
following more general theorem.

\begin{teorema}\label{T7.2} Let $K=k(y)$. Then the 
following statements are equivalent:
\l
\item $E=k(z)$ with $k\subseteq E\subseteq K$,
$[E:k]=p^m$ such that $g(z)=\chi\in k$ for some
$\chi\in k$ and for some additive polynomial $g(X)$ such that
$g(X)\mid f(X)$, that is, $g=f_V$ for an additive subgroup $V$ of
$\G_f$ of dimension $m$ over ${\ma F}_p$,

\item
there exist $A_{n-1}, A_{n-2}, \ldots,A_1,A_0\in \G_f$, 
$C\in k$  and a ${\ma F}_p$--vector subspace
$\H$ of $\G_f$ of dimension $n-m$ such that
\las
\item $z$ satisfies
\begin{gather}\label{Eq7.2}
z=\elemental Ayn +C,
\end{gather}

\item
for $\beta\in \G_f$,
\[
\elemental A{\beta}n=0\iff \beta\in \H.
\]
\end{list}
\end{list}

The relation between $E$ and  {\rm{(\ref{Eq7.2})}} is given as
follows. If $\H=\L\{\mu_{m+1},\cdots,\mu_{n}\}$ where $\{\mu_1,\ldots,
\mu_n\}$ is the basis of $\G_f$ such that if $G=\langle \sigma_1,
\ldots,\sigma_n\rangle$, $\sigma_i (y)=y+\mu_i$, $1\leq i\leq n$, 
then $E$ is the fixed field by the subgroup $H:=\langle \sigma_{m+1},
\ldots,\sigma_{n}\rangle$ of $G$. That is, $H$ corresponds to
$\H$ under the isomorphism given in Proposition {\rm{\ref{P2.2}}}
and if $\G_f=\H\oplus V$, $g(X)=f_V(X)=\prod_{\delta\in 
V}(X-\delta)\mid f(X)$.
\end{teorema}

\begin{proof}
Let $G=\Gal(K/k)=\langle \sigma_1,\ldots, \sigma_n\rangle$ with
$\sigma_i(y)=y+\mu_i$, where 
$\mu_i\in\G_f$ and $\{\mu_1,\ldots,\mu_n\}$ is a
$\ma F_p$--basis of $\G_f$. More precisely, we have that if
$\sigma_1,\ldots,\sigma_n\in G$ with $\sigma_i(y)=y+\mu_i$, 
then $G=\langle\sigma_1,\ldots,
\sigma_n\rangle\iff \{\mu_1,\ldots,\mu_n\}$ is a basis of
 $\G_f/{\ma F}_p$. Note that for $0\leq \alpha_i\leq
p-1$, $1\leq i\leq n$, $\sigma=\sigma_1^{\alpha_1}\cdots\sigma_n^{
\alpha_n}$ we have $\sigma(y)=y+\sum_{i=1}^n\alpha_i\mu_i$.

First we consider a subfield $E$ of $K$ of degree $p^m$ over $k$.
We may choose a set of $n$ generators of $G=\Gal(K/k)=
\langle\sigma_1,\ldots,\sigma_n\rangle$ in such a way that
$E=K^{\langle\sigma_{m+1},\ldots,\sigma_n\rangle}$ 
is the fixed field under
$H=\langle\sigma_{m+1},\ldots,\sigma_n\rangle$. We have
$\Gal(K/E)= \langle\sigma_{m+1},\ldots,\sigma_n\rangle$.

Let $\theta\colon G\to\G_f$ be the isomorphism given by $\sigma_i
\mapsto \mu_i$, $1\leq i\leq n$. Let
$\H=\theta(H)<\G_f$ and let $V$ be an arbitrary section of the exact 
sequence
\[
0\lra \H\stackrel{i}{\lra}\G_f\stackrel{\pi}{\lra}\G_f/\H\lra 0,
\]
that is, $V=\varphi(\G_f/\H)<\G_f$ where $\varphi\colon \G_f/\H
\lra\G_f$ satisfies $\pi\circ\varphi =\Id_{\G_f/\H}$. We have $\G_f\cong
\H\oplus V$ as ${\ma F}_p$--vector spaces. 

From Theorem \ref{T5.4}, we have that there exist $z\in E$ and $\chi\in k$ 
such that $E=k(z)$ with $f_V(z)=\chi\in k$ and $f_V(X)=\prod_{\delta\in V}
(X-\delta)\mid \prod_{\delta\in\G_f}(X-\delta)=f(X)$.

We have that $\Gal(E/k)\cong\langle\bar{\sigma}_1,\ldots,\bar{\sigma}_m
\rangle$ where $\bar{\sigma}_i=\sigma_i|_E$ or, equivalently,
 $\bar{\sigma}_i=
\sigma_i\bmod \Gal(K/E)$. 

Now, let $\sigma_i(z)=z+\gamma_i$, $1\leq i\leq m$, where 
$\{\gamma_1,\ldots,\gamma_m\}$ is a basis of $V$ and
$\sigma_j(z)=z$ for $m+1\leq j\leq n$. We denote
$\gamma_j=0$ for $m+1\leq j\leq n$.

Let $A_{n-1},A_{n-1},\ldots, A_1, A_0\in \G_f$ be arbitrary and let
\begin{equation}\label{Eq7.4}
w:=\elemental Ayn.
\end{equation}
That is, if we denote $l(X)=\elementalito AXn$, then $w=l(y)$.

We will prove that there exist $A_{n-1},A_{n-2},\ldots, A_1, A_0
\in \G_f$ and $D\in k$ such that 
\begin{equation}\label{Eq7.3}
z=w+D.
\end{equation}

We have
\begin{equation}\label{Eq7.5}
\sigma_i(w)=\sigma_i(l(y))=l(\sigma_i(y))=l(y+\mu_i)=l(y)+l(\mu_i)=w+l(\mu_i).
\end{equation}

It follows that
\begin{gather}
\sigma_i(w)=w+\gamma_i,\quad 1\leq i\leq n\iff l(\mu_i)=\gamma_i,
\quad 1\leq i\leq n \label{Eq7.6}\\
\iff M\left[\begin{array}{c}A_0\\A_1\\ \vdots\\ A_{n-2}\\
A_{n-1}\end{array}\right]=\left[\begin{array}{c}\gamma_1\\ \gamma_2\\ \vdots\\ 
\gamma_{n-1}\\ \gamma_n\end{array}\right]=
\left[\begin{array}{c}\gamma_1\\ \vdots\\ 
\gamma_{m}\\0 \\ \vdots \\ 0\end{array}\right], \nonumber
\end{gather}
where $M$ is the matrix
\[
M=\matriz{\mu_1}\mu.
\]

We will prove that $M$ is non--singular. Let
\[
B(X):=\matriz X\mu=\left[\begin{array}{c}F(X)\\F(\mu_2)\\ \vdots\\ 
F(\mu_{n-1})\\F(\mu_n)\end{array}\right]
\]
where $F(Z):=[Z\quad Z^p\ \cdots\  Z^{p^{n-1}}\quad Z^{p^n}]$ 
with $Z\in\{X,\mu_2,\ldots,\mu_n\}$ denotes
the rows of $B(X)$. We have
$B(\mu_1)=M$ and $\det B(X)$ is an additive polynomial in
$k[X]$ of degree $p^{n-1}$.

Let $\{i_2,\ldots,i_n\}\subseteq \ma F_p^{n-1}$ and $\xi=i_2\mu_2+
\cdots+i_n\mu_n$. Then:
\[
B(\xi)=\left[\begin{array}{c}F(i_2\mu_2+\cdots+i_n\mu_n)\\F(\mu_2)\\ \vdots\\ 
F(\mu_{n-1})\\F(\mu_n)\end{array}\right]=
\left[\begin{array}{c}i_2F(\mu_2)+\cdots+i_nF(\mu_n)\\F(\mu_2)\\ \vdots\\ 
F(\mu_{n-1})\\F(\mu_n)\end{array}\right].
\]
Therefore $\det B(\xi)=0$ for every $\xi\in\{i_2\mu_2+\cdots +
i_n\mu_n\mid i_2,\ldots,i_n\in\ma F_p\}=C$. Since $\{\mu_2,\ldots,
\mu_n\}$ is a linearly independent set over ${\ma F}_p$,
it follows that $|C|=p^{n-1}=\deg B(X)$. In this way, we obtain that
$C$ is the set of roots of $\det B(X)$. In particular, since
$\mu_1\notin C$, $\det B(\mu_1)=\det M\neq 0$ and $M$ is non--singular.

Hence (\ref{Eq7.6}) has a unique solution:
\begin{equation}\label{Eq7.7}
\left[\begin{array}{c}A_0\\A_1\\ \vdots\\ A_{n-2}\\
A_{n-1}\end{array}\right]=M^{-1}\left[\begin{array}{c}\gamma_1\\ \vdots\\ 
\gamma_{m}\\0 \\ \vdots \\ 0\end{array}\right].
\end{equation}

Let $\beta=\sum_{i=1}^n c_i\mu_i\in \G_f$ with $c_i\in{\ma F}_p$,
$1\leq i\leq n$.
Therefore 
\[
l(\beta)=l(\sum_{i=1}^nc_i\mu_i)=\sum_{i=1}^nl(c_i\mu_i)=
\sum_{i=1}^nc_il(\mu_i)=\sum_{i=1}^nc_i\gamma_i=\sum_{i=1}^mc_i\gamma_i.
\]
It follows that $l(\beta)=0\iff c_1=\ldots=c_m=0\iff \beta\in\L\{\mu_{m+1},
\ldots,\mu_n\}=\H$.

Finally, we have $\sigma_i(z-w)=z-w$ for all $1\leq i\leq n$, so that
$z-w=D\in k$ and $z$ is in the form (\ref{Eq7.2}).

To prove the converse, let $z$ be given by (\ref{Eq7.2}), $z=l(y)+D$. Then
\[
\sigma_i(z)=\sigma_i(l(y)+D)=l(\sigma_i(y))+D=l(y+\mu_i)+D=
l(y)+l(\mu_i)+D=z+l(\mu_i)
\]
and we have $l(\mu_i)=0\iff i\geq m+1$. Therefore $k(z)\subseteq K^{\langle
\sigma_{m+1},\ldots,\sigma_n\rangle}$. Now, for any $c_1,\ldots,
c_m\in {\ma F}_p$, not all equal to zero, $\sigma_1^{c_1}\cdots\sigma_m^{c_m}(z)
=z+l(\beta)$ with $\beta=\sum_{i=1}^m c_i\mu_i\neq 0$, $l(\beta)\neq 0$. This
implies that $[k(z):k]\geq p^m$.
It follows that $[k(z):k]=p^m$ and that $k(z)=K^{\langle \sigma_{m+1},\ldots,
\sigma_n\rangle}$.

Let $\xi_i:=l(\mu_i)$, $1\leq i\leq m$, $V=\L\{\xi_1,\ldots,\xi_m\}$ and
 $f_V(X)\mid f(X)$. Then
\begin{gather*}
f_V(z)=f_V(l(y)+D)=f_V(\sum_{i=0}^{n-1}A_iy^{p^i}+D)=
\sum_{i=0}^{n-1} f_V(A_iy^{p^i})+f_V(D)=\chi.
\end{gather*}

We will prove that $\chi=f_V(z)\in k$.

If $\sigma:=\sigma_1^{c_1}\cdots\sigma_n^{c_n}\in G=
\Gal(K/k)$ then, for $\mu:=\sum_{i=1}^n c_i\mu_i$, $\sigma(y)=y+\mu$, 
we have
\begin{align*}
\sigma(f_V(z))&=\sigma(f_V(l(y)+D))=f_V(l(y+\mu)+D)=f_V(l(y)+l(\mu)+D)\\
&=f_V(l(y)+D)+f_V(l(\mu))=f_V(z)+f_V(l(\mu)).
\end{align*}
Finally,
\[
l(\mu)=\sum_{i=1}^n c_il(\mu_i)=\sum_{i=1}^m c_i\gamma_i\in V,
\]
so that $f_V(l(\mu))=0$ and $\sigma(f_V(z))=f_V(z)$ for all $\sigma\in
\Gal(K/k)$. It follows that $\chi=f_V(z)\in k$.
This finishes the proof.
\end{proof}

\section{Multicyclic extensions}\label{S8}

In this section we are interested in abelian extensions with Galois
group isomorphic to $\Big({\ma Z}/p^m{\ma Z}\Big)^n$. 
To this end, we generalize the results of previous sections and
consider extensions using the Witt ring.
First we fix some notation that will be used in this part of the paper.
For details, the main sources are the papers of Witt
\cite{Wit36-2} and of Schmid \cite{Sch36}. A summary can be
found in \cite[Cap\'itulo 11]{RzeVil2014}.

Given a commutative ring with identity $R$,
$W_m(R)$ denotes the ring of Witt vectors over $R$ of length $m$, that
is, $\vec \alpha=\Wittc \alpha m\in W_m(R)$ means that
$\alpha_i\in R$, $1\leq i\leq m$ and that $\alpha^{(i)}=
\alpha_1^{p^{i-1}}+p\alpha_2^{
p^{i-2}}+\cdots + p^{i-1} \alpha_i$ 
are the ghost components of $\vec \alpha$. Witt operations are denoted
by $\Witt +, \Witt -, \Witt \cdot$. Let $\vec 1=\big(1,0,\ldots,0\big)$ 
and for $t\in{\ma N}$, $\vec t=
{\underbrace{\vec 1\Witt + \cdots  \Witt +\vec 1}_{t\text{\ times}}}$.
We have $\vec p^j\Witt \cdot \vec 1=\big(\underbrace{0,0,\ldots,0}_{j}, 
\underbracket[0pt]{1}_{\substack{\uparrow\\ j+1}},0,\ldots, 0\big)$.
Furthermore $W_m({\ma F}_p)\cong {\ma Z}/p^m{\ma Z}$. 

For $u\in R$ we write $\{u\}:=\big(u,0,\ldots,0\mid 
u, u^p,\ldots, u^{p^{m-1}}\big)$.
For a vector $\vec x=\v x m \big)\in W_m(R)$, we define
$\vec x^p:=\big(x_1^p,\ldots, x_m^p\big)$. Note that
$\vec x^p$ is not the $p$--power of $\vec x$ with Witt
multiplication, that is, $\vec x^p\neq \underbrace{\vec x\Witt \cdot
\cdots \Witt \cdot \vec x}_{p}$. The Artin--Schreier--Witt operator
is defined by $\wp(\vec x):=\vec x^p\Witt -\vec x$. 

Let $k$ be a field of characteristic $p$. Then
$\wp(\vec x)=\vec 0\iff \vec x\in W_m({\ma F}_p)\subseteq W_m(k)$.
We have $(\vec x\Witt \circ \vec y)^p= \vec x^p\Witt \circ \vec y^p$
where $\circ\in\{+,-,\cdot\}$.
Let $K/k$ be a cyclic extension of degree $p^m$. Here we write,
for $\vec y=\v y m\big)$, $\sigma \vec  y=\v {\sigma y}m\big)$.
Let $G=\Gal(K/k)=\langle\sigma\rangle$,
$o(\sigma)=p^m$. Now $\vec 1\in W_m(K)$ satisfies 
\[
\Tr_{K/k}\vec 1= \Witt \sum_{\sigma\in G}\sigma \vec 1
=\vec p^m\Witt\cdot \vec 1=\vec p^m=\vec 0.
\]
Thus there exists $\vec y\in W_m(K)$ such that $(\sigma\Witt -\vec 1)
(\vec y)=\vec 1$, that is, $\sigma\vec y=\vec y\Witt +\vec 1$. 

Let $\vec x:=\wp \vec y=\vec y^p\Witt - \vec y$. Then
\[
\sigma(\wp \vec y)=\wp (\sigma \vec y)=\wp (\vec y\Witt + \vec 1)=
\wp(\vec y)\Witt +\wp(\vec 1)=\wp(\vec y),
\]
so, we obtain that $\vec x\in W_m(k)$. Now, $\sigma(\vec y)=
\vec y\Witt + \vec 1$ where $k(\vec y):=k\v ym\big)\subseteq K$
and $k(\vec y)/k$ is a cyclic extension of degree $p^m$ 
because $\sigma^t\vec y=\vec y\Witt +\vec t$ and
$p^m$ is minimum satisfying $\vec p^m=\vec 0$. 
Therefore $K=k(\vec y)$.

In general, if $k$ is a field of characteristic $p$, the vector
$\vec y \in W_m(k)$ is invertible $\iff y_1\neq 0$.

The following Proposition generalizes Proposition \ref{P7.0}.

\begin{proposicion}\label{P8.2}
If $k(\vec y_1)$ and $k(\vec y_2)$ are cyclic extensions of degree
$p^m$ of $k$ with $\vec y_1,\vec y_2\in W_m(K)$
and $\wp (\vec y_i)=\vec x_i\in W_m(k)$,
$i=1,2$, then the following statements are equivalent: 
\l
\item  $k(\vec y_1) =k(\vec y_2)$,
\item there exist $\vec j\in W_m({\ma F}_p)$ invertible, that is,
$\mcd (j,p)=1$, and $\vec z\in W_m(k)$ such that $\vec y_1=\vec j\Witt \cdot
\vec y_2\Witt + \vec z$ and $\vec x_1=\vec j\Witt \cdot \vec x_2\Witt +
\wp (\vec z)$. \hfill \fin
\end{list}
\end{proposicion}

Let $q=p^n$ and consider a field $k$ such that $\F\subseteq k$. 
Let $\vec x^q:=\big(x^q_1,\ldots, x^q_m
\big)$. Then $\vec x^q\Witt - \vec x=0\iff
\vec x\in W_m(\F)\subseteq W_m(k)$. 
The ring $W_m(\F)$ is known as
a {\em Galois ring}. As group, we see that
$W_m(\F)$ is a free $W_m({\ma F}_p)$--module of rank $n$.
In particular, $W_m(\F)\cong\Big({\ma Z}/p^m{\ma Z}\Big)^n$
as groups.

\begin{proposicion}\label{P8.1} We have that $W_m(\F)$ 
is a free $W_m({\ma F}_p)$--module of rank $n$, where $q=p^n$.
More precisely, let $\{\mu_1,\ldots,\mu_n\}$ be a basis of
$\F$ over ${\ma F}_p$ and set $\vec \mu_i:=\{\mu_i\}=
\big(\mu_i,0,\ldots,0\big)$, $1\leq i\leq n$. Then $\{\vec \mu_1,
\ldots,\vec \mu_n\}$ is a $W_m({\ma F}_p)$--basis of $W_m(\F)$.
That is
\[
W_m(\F)=\bigoplus_{i=1}^n W_m({\ma F}_p) \cdot \vec \mu_i.
\]
\end{proposicion}

\begin{proof}
Let $\vec \alpha_1,\ldots, \vec\alpha_n\in W_m({\ma F}_p)$,
with $\vec \alpha_i=\v{\alpha_i} m\mid \vWitt{\alpha_i} m$, 
$\vec \mu_i=\big(\mu_i,0\ldots,0\mid
\mu_i,\mu_i^p,\ldots,\mu_i^{p^{m-1}}\big)$. We have that
$\vec \alpha_i\Witt\cdot \vec\mu_i=\big(\alpha_{i1}\mu_i,\ldots,?\mid
\alpha_{i1}\mu_i,\ldots, ?\big)$. If $\Witt \sum_{i=1}^n
\vec\alpha_i\vec\mu_i=
\big(\ldots\mid \sum_{i=1}^n \alpha_{i1}\mu_i,\ldots\big)
=\vec 0=\big(0,\ldots,0\mid 0,\ldots, 0\big)$, 
we obtain that $\sum_{i=1}^n
\alpha_{i1}\mu_i=0$ which implies $\alpha_{i1}=0$ for all
$1\leq i\leq n$.

We obtain $\vec \alpha_i=(0,\alpha_{i2},\ldots,\alpha_{im}\mid
0,p\alpha_{i2},\ldots)$. Therefore, the second ghost component of
 $\Witt\sum_{i=1}^n \vec\alpha_i\Witt\cdot \vec\mu_i=\vec 0$ is
$\vec 0^{(2)}=\sum_{i=1}^n p\alpha_{i2}\mu_i^p$. Thus
\[
0=\vec 0_2=\frac{1}{p}\big(\vec 0^{(2)}-\vec 0_1^p\big)=
\sum_{i=1}^n\alpha_{i2}\mu_i^p=\sum_{i=1}^n\alpha_{i2}^p\mu_i^p=
\Big(\sum_{i=1}^n\alpha_{i2}\mu_i\Big)^p=0.
\]
Hence $\sum_{i=1}^n\alpha_{i2}\mu_i=0$ so that
$\alpha_{i2}=0$ for all $1\leq i\leq n$. Keeping on with this
procedure, we obtain $\vec \alpha_i=\vec 0$
and the result follows.
\end{proof}

\begin{observacion}\label{O8.1'}{\rm{In general 
we have that $\{\vec \xi_1,\ldots,\vec 
\xi_n\}$ is a basis of
$W_m(\F)$ over $W_m({\ma F}_p)$ if and only if
$\{\xi_{11},
\ldots,\xi_{n1}\}$ is a basis of $\F$ sobre ${\ma F}_p$. 
This can be proved
following the proof of Proposition \ref{P8.1} and noting that
if $\vec\alpha_1,\ldots,\vec\alpha_n\in W_m({\ma F}_p)$, then
$\Witt \sum_{i=1}^n \vec\alpha_i \Witt\cdot \vec\xi_i=\big(
\sum_{i=1}^n \alpha_{i1}\xi_{i1},\ldots \big)$, etc. 
We leave the details of the proof to the interested reader.
}}
\end{observacion}

We consider the equation $\vec y^q\Witt - \vec y=\vec \alpha$ where
$\vec \alpha\in W_m(k)$. Let $\vec y_0\in W_m(\bar{k})$ be a solution
of $\vec y^q\Witt - \vec y=\vec\alpha$, where $\bar{k}$
denotes an algebraic closure of $k$. Note that if
$K=k(\vec y)=k\v ym\big)$ is a cyclic extension of degree $p^m$ over
$k$, then $y_m\notin k\v y{m-1}\big)$ since otherwise we would have that
$[K:k]=[k\v y{m-1}\big):k]\leq p^{m-1}$. So, $K=k(y_m)$. The set of roots
of $\vec y^q\Witt -\vec y=\vec \alpha$ is the set $\{\vec y_0\Witt +
\vec \mu\}_{\vec \mu\in W_m(\F)}$. Let $K=k(\vec y_0)$. Then, because
$\F\subseteq k$, we have $W_m(\F)\subseteq W_m(k)$ and therefore
$K/k$ is a normal extension. Since $|W_m(\F)|=q^m=p^{nm}$, all the
roots of $\vec y^q\Witt -\vec y=\vec \alpha$ are different and therefore
$K/k$ is a Galois extension.

Let $G:=\Gal(K/k)$ and $\sigma\in G$. Then $\vec y_0$ and $\sigma
(\vec y_0)$ are conjugate so there exists $\vec \xi\in W_m(\F)$ such that
$\sigma (\vec y_0)=\vec y_0\Witt + \vec\xi$. Put $\sigma_{\vec \xi}=
\sigma$.

\begin{proposicion}\label{P8.3} With the above notation, we
have that $\varphi\colon G\lra W_m(\F)$, given by
$\varphi(\sigma_{\xi})=\xi$ is a group monomorphism. 
This implies that the extension $K/k$
is abelian and that $G\subseteq W_m(\F)\cong
\Big({\ma Z}/p^m{\ma Z}\Big)^n$. Therefore
$G\cong{\ma Z}/p^{a_1}{\ma Z}\times\cdots\times
{\ma Z}/p^{a_n}{\ma Z}$ with
$m\geq a_1\geq a_2\geq \cdots\geq a_n\geq 0$. \fin
\end{proposicion}

Conversely, let $K=k(z)$ be a finite abelian $p$--extension of
exponent $m$ and rank $n$, that is 
$G\cong{\ma Z}/p^{a_1}{\ma Z}\times\cdots\times
 {\ma Z}/p^{a_n}{\ma Z}$ with
$m\geq a_1\geq a_2\geq \cdots\geq a_n\geq 1$ 
and we assume that $\F\subseteq k$ where
$q=p^n$. Let $K=k(z_1,\ldots,z_n)$ with
$\Gal(k(z_i)/k)\cong {\ma Z}/p^{a_i}{\ma Z}$, $1\leq i\leq n$.
Then, there exists $\vec y_i$ such that
$k(z_i)=k(\vec y_i)$ with $\vec y_i^p\Witt -\vec y_i=\vec \alpha_i\in
W_m(k)$, $\vec y_i\in W_m(K)$, where we write $\vec y_i=\big(
\underbrace{0,\ldots,0}_{m-a_i},y_{i,m-a_i+1},\ldots, y_{i,m}\big)$, 
that is, we complete with zeros the components to make the
vectors of length $m$.

Let $G=\langle\sigma_1,\ldots,\sigma_n\rangle$ with
$\sigma_j(\vec y_i)=\begin{cases}\vec y_i\Witt +\vec 1&\text{if $i=j$}\\
\vec y_i&\text{if $i\neq j$}\end{cases}$, and $o(\sigma_i)=p^{a_i}$.
We define
\[
\vec y:=\vec \xi_1\Witt \cdot \vec y_1\Witt +\cdots \Witt +\vec\xi_n
\Witt \cdot \vec y_n,
\]
where $\{\vec\xi_1,\ldots,\vec\xi_n\}$ is a basis of $W_m(\F)$
over $W_m({\ma F}_p)$. Then $\vec y_i\in W_m(K)$, 
$\vec \xi_i\in W_m(\F)\subseteq W_m(K)$.
It follows that $k(\vec y)\subseteq K$. We will 
show that $k(\vec y)=K$.

Let $\sigma\in G$, say $\sigma=
\sigma_1^{b_1}\cdots \sigma_n^{b_n}$,
$0\leq b_i\leq a_i-1$, $1\leq i\leq n$. Then,
writing $\vec b_i=\{b_i\}$,
\begin{align*}
\sigma\vec y&=\sigma\Big(\Witt \sum_{i=1}^n\big(
\vec \xi_i\Witt \cdot
\vec y_i\big)\Big)=\Witt \sum_{i=1}^n\sigma\big(
\vec\xi_i\Witt \cdot \vec y_i\big)\\
&\underbracket[0pt]{=}_{\substack{\uparrow\\
 \vec \xi_i\in W_m(\F)}}
\Witt \sum_{i=1}^n\vec \xi_i\Witt \cdot \sigma(\vec y_i)=\Witt \sum_{i
=1}^n\vec \xi_i\Witt\cdot (\vec y_i\Witt +\vec b_i)\\
&=\Witt \sum_{i=1}^n\vec \xi_i\Witt \cdot \vec y_i\Witt +
\Witt \sum_{i=1}^n \vec b_i\Witt \cdot \vec\xi_i =\vec y\Witt + \vec \xi,
\end{align*}
where $\vec \xi:=
\Witt\sum_{i=1}^n\vec b_i\Witt\cdot \vec \xi_i\in
 W_m(\F)$. Therefore
$\sigma\vec y=\vec y\iff \vec \xi=\vec 0\iff 
\vec b_1=\cdots=\vec b_n=
\vec 0\iff b_1=\cdots=b_n=0\iff \sigma=\Id$. 
It follows that $K=k(\vec y)$.

In brief, we have:

\begin{teorema}\label{T8.4} Let $k$ be a 
field of characteristic
$p>0$ such that $\F\subseteq k$ with $q=p^n$. Let
$\vec\alpha\in W_m(k)$. If
$K=k(\vec y_0)$ where $\vec y_0$ is a
root of $\vec y^q\Witt -\vec y=\vec \alpha\in 
W_m(k)$, then $K/k$ is an abelian $p$--extension
of exponent $p^h$ with $h\leq m$ and rank
$l$ with $l\leq n$. 
Furthermore, if $G:=\Gal(K/k)$, we have that
$G$ is isomorphic in a natural way to a
subgroup of the Galois ring $W_m(\F)$.

Conversely, if $K=k(\vec y_0)$ is an abelian $p$--extension
of exponent $p^h$ and rank $l$, then $\vec y_0$ is root of
some equation of the form $\vec y^q\Witt -\vec y=
\vec \alpha$ for some $\vec \alpha\in W_m(k)$. \fin
\end{teorema}

Now consider a finite abelian $p$--extension $K/k$ given by
$\vec y^q\Witt -\vec y=\vec\alpha\in W_m(k)$ 
and where we assume that $\F\subseteq k$. Say that
$G=\Gal(K/k)\cong \prod_{i=1}^n{\ma Z}/
p^{a_i}{\ma Z}$ with $m\geq a_1\geq a_2\geq 
\cdots\geq a_n\geq 0$. Let $\vec \xi\in W_m(\F)$ and let
\begin{gather}\label{Eq**}
\vec y_{\vec\xi}:=(\vec \xi^{p^{n-1}}\Witt \cdot \vec y^{p^{n-1}})\Witt +
(\vec \xi^{p^{n-2}}\Witt \cdot \vec y^{p^{n-2}})\Witt +\cdots \Witt +
(\vec \xi^{p}\Witt \cdot \vec y^{p})\Witt +(\vec\xi\Witt\cdot \vec y).
\end{gather}
Then
\[
\vec y_{\vec\xi}^p\Witt -\vec y_{\vec\xi}=\vec \xi\Witt\cdot \vec \alpha,
\]
that is, $k(\vec y_{\vec\xi})/k$ is a cyclic extension of
degree $p^h$ with $h\leq m$.

From now on, we will assume that $G:=\Gal(K/k)=
\langle\sigma_1,\ldots,\sigma_n\rangle \cong \Big(
{\ma Z}/p^m{\ma Z}\big)^n$ with $o(\sigma_i)=p^m$ for all 
$1\leq i\leq n$. The group $G$ has 
$\frac{q^m-q^{m-1}}{p^m-p^{m-1}}$ distinct
cyclic subgroups of order $p^m$. In particular, we must have
$\alpha_1\neq 0$. We want to verify that among all
the extensions $k(\vec y_{\vec\xi})$ are all the
cyclic subextensions of degree $p^m$. 
Under the isomorphism $G\cong W_m(\F)$, 
we consider $\sigma_i:=\sigma_{\vec \xi_i}$, 
$1\leq i\leq n$ where $\{\xi_1,\ldots,\xi_n\}$ 
is a basis of $\F$ sobre ${\ma F}_p$ and where
we recall that $\vec\xi_i=\{\xi_i\}=(\xi_i,0,\ldots, 0)$.
In fact we note that $\{\xi_1,\ldots,\xi_n\}$ 
is a basis $\F$ over ${\ma F}_p$ if and only if $G=
\langle\sigma_1,\ldots,\sigma_n\rangle$.

It is easy to see that for $\sigma_{\vec\delta}\in G$
we have $\sigma_{\vec\delta}(\vec y_{\vec\xi}) =
\vec y_{\vec\xi}\Witt +\Witt \sum_{i=0}^{n-1}
(\vec \xi\Witt \cdot \vec \delta)^{p^i}$. Then
$\sigma_{\vec\delta}(\vec y_{\vec\xi})=\vec y_{\vec\xi}\iff
\Witt \sum_{i=0}^{n-1} (\vec \xi\Witt \cdot \vec \delta)^{p^i}=\vec 0$.

In general, let $g_{\vec \xi}(\vec\delta):=\Witt \sum_{i=0}^{n-1}
(\vec \xi\Witt \cdot \vec \delta)^{p^i}$. Then
$g_{\vec \xi}(\vec \delta)^p\Witt - g_{\vec \xi}(\vec \delta)=\vec 0$,
that is, $g_{\vec \xi}(\vec \delta)\in W_m({\ma F}_p)$. The map
$g_{\vec\xi}\colon W_m(\F)\lra W_m({\ma F}_p)$ 
is not surjective in general and we have
\[
\frac{W_m(\F)}{\ker g_{\vec \xi}}\cong \im g_{\vec\xi}\subseteq W_m(
{\ma F}_p),\quad |W_m({\ma F}_p)|=p^m,
\]
so that $|\ker g_{\vec \xi}|\geq \frac{|W_m(\F)|}{|W_m({\ma F}_p)|}=
\frac{q^m}{p^m}$.

Observe $k(\vec y_{\vec\xi})$ is the fixed field of $K$ by
$\ker g_{\vec \xi}$.

\begin{proposicion}\label{P8.5} With the notation as above
and $\vec\xi=(\xi_1,\ldots,\xi_m)$,
we have that $[k(\vec y_{\vec\xi}):k]=p^m\iff \vec
\xi$ is invertible in $W_m(\F) \iff \xi_1\neq 0$.
\end{proposicion}

\begin{proof} In the expression $\Witt \sum_{i=0}^{n-1}
(\vec \xi\Witt \cdot \vec \delta)^{p^i}$, the first component is
$\sum_{i=0}^{n-1}(\xi_1\delta_1)^{p^i}$.
Assume that $\vec \xi$ is invertible, that is, $\xi_1\neq 0$.
Consider the map $\psi\colon\F\lra{\ma F}_p$ given by
$\psi(\delta)=\sum_{i=0}^{n-1}(\xi_1 \delta)^{p^i}$. This
map is not zero since the polynomial 
$p(x)=(\xi_1x)^{p^{n-1}}+\cdots+(\xi_1x)^p+
(\xi_1 x)=0$ has $p^{n-1}$ roots. Therefore, if we
consider the extension $k(y_{\xi_1})/k$ given by
$y_{\xi_1}^p-y_{\xi_1}=\alpha_1\neq 0$,
the group fixing the extension is not all $\F$ and in particular
$[k(y_{\xi_1}):k]=p$. Therefore $[k(\vec y_{\vec \xi}):k]=p^m$.

Conversely, in case $[k(\vec y_{\vec \xi}):k]=p^m$, necessarily
$[k(y_{\xi_1}):k]=p$ and the argument is reversible and
it follows $\vec\xi$ is invertible. 
\end{proof} 

\begin{corolario}\label{C8.5'} The cyclic subextensions
of degree $p^m$ are given by $k(\vec y_{\vec\xi})$ 
where $\vec y_{\vec\xi}$ is given by
{\rm{(\ref{Eq**})}}, where $\vec\xi$ is invertible
and we have
$$
\vec y^p_{\vec\xi}\Witt -\vec y_{\vec \xi}=\vec\xi\Witt\cdot\vec\alpha.
\eqno{\fin}
$$
\end{corolario}

In particular, taking a basis $\{\vec\mu_1,\ldots,\vec\mu_n\}$ of $W_m(\F)$
over $W_m({\ma F}_p)$, we have that $K=k(\vec y)=k(\vec y_{\vec \mu_1},
\ldots, \vec y_{\vec \mu_n})$.

\begin{proposicion}\label{P8.5''} Let $k$ be a field such that
$\F\subseteq k$.
Let $K/k$ be an abelian extension
with Galois group isomorphic to $W_m(\F)$. 
Assume that $K=k(\vec z_1,\ldots,\vec z_n)$ 
with $\vec z_i\in W_m(K)$, 
$\Gal(k(\vec z_i)/k)\cong {\ma Z}/p^m{\ma Z}$, $1\leq i\leq n$.
Then all the subextensions $k\subseteq k(\vec z)\subseteq K$
such that $\Gal(k(\vec z)/k)\cong {\ma Z}/p^m{\ma Z}$ 
are given by
\[
\vec z=\Witt \sum_{i=1}^n\vec \alpha_i\Witt\cdot \vec z_i,
\]
with $\vec \alpha_i\in W_m({\ma F}_p)$, $1\leq i\leq n$
and some $\vec \alpha_{i_0}$ invertible.
\end{proposicion}

\begin{proof}
Let $G:=\Gal(K/k)=\langle\sigma_1,\ldots,\sigma_n\rangle$ 
such that $\sigma_i\vec z_j=\vec z_j\Witt +\vec \delta_{ij}$
with $\vec \delta_{ij}=
\begin{cases} \vec 1&\text{if $i=j$}\\ 
\vec 0&\text {if $i\neq j$}\end{cases}$.

Let $\vec \alpha_1,\ldots,\vec \alpha_n\in W_m({\ma F}_p)$ and let
$\vec z =\Witt \sum_{i=1}^n\vec \alpha_i\Witt\cdot \vec z_i$. Let
$\wp(\vec z_i)=\vec z_i^p\Witt -\vec z_i=\vec \gamma_i$ with 
$\vec \gamma_i=\big(\gamma_{i1},\ldots,\gamma_{im}\big)$ and
$\gamma_{i1}\notin \wp(k)$. Then
\[
\wp(\vec z)=\Witt \sum_{i=1}^n\vec \alpha_i\Witt\cdot \wp(\vec z_i)
=\Witt \sum_{i=1}^n\vec \alpha_i\Witt\cdot \vec \gamma_i=:\vec \gamma,
\]
with $\gamma_1=\sum_{i=1}^n\alpha_{i1}\gamma_{i1}$. We have
$[k(\vec z):k]=p^m\iff \gamma_1\notin \wp(k)$.

Now assume $\wp(\vec z)=\Witt \sum_{i=1}^n\vec 
\alpha_i\Witt\cdot \vec \gamma_i=\vec \gamma=\wp(\vec A)$ for some
$\vec A\in W_m(k)$. Then $\wp(\vec z\Witt - \vec A)=\vec 0$, 
that is, $\vec z\Witt - \vec A\in W_m({\ma F}_p)$, $\vec z=\vec \beta
\Witt + \vec A$ with $\vec \beta\in W_m({\ma F}_p)$.
In this case, if existed $\vec \alpha_{i_0}$ invertible, then
$\vec z=\Witt \sum_{i=1}^n\vec \alpha_i\Witt\cdot \vec z_i=\vec \beta
\Witt + \vec A$ so that
\begin{gather*}
\vec z_{i_0}=\Witt - \Witt \sum_{\substack{i=1\\ i\neq i_0}}^n
\vec \alpha_{i_0}^{-1}\Witt \cdot \vec \alpha_i\Witt\cdot \vec z_i\Witt +
\vec \alpha_{i_0}^{-1}\Witt \cdot \vec \beta\Witt + \vec \alpha_{i_0}^{-1}
\Witt \cdot \vec A.
\intertext{Now, because 
$\vec \beta\in W_m({\ma F}_p)\subseteq W_m(k)$ and
$\vec A\in W_m(k)$, it follows that $\vec z_{i_0}\in k(\vec z_1,\ldots,
\vec z_{i_0-1},\vec z_{i_0+1},\ldots, \vec z_n)$, that}
K=k(\vec z_1,\ldots,\vec z_n)=k(\vec z_1,\ldots,
\vec z_{i_0-1},\vec z_{i_0+1},\ldots, \vec z_n)
\end{gather*}
and that $[K:k]\leq p^{m(n-1)}<p^{mn}$, which is absurd.

In summary, if some $\vec \alpha_{i_0}$ were invertible, $\vec \gamma=
\wp(\vec z)\notin \wp(W_m(k))$.

With this procedure we obtain $t$ extensions $k(\vec z)$ with 
$[k(\vec z):k]=p^m$ and $t=\big|\{(\vec\alpha_1,\ldots,\vec\alpha_n)\mid
\vec\alpha_i\in W_m({\ma F}_p)\text{\ and some $\vec \alpha_i$ invertible}
\}\big|$.

We have $\vec\alpha_1,\ldots,\vec\alpha_n\in W_m({\ma F}_p)$ are
non invertible if and only if $\alpha_{11}=\alpha_{21}=\cdots=\alpha_{n1}
=0$ where $\vec \alpha_i=(\alpha_{i1},\ldots,\alpha_{im})$, $1\leq i\leq n$.
Thus $t=\big|W_m({\ma F}_p)^n\big|-\big|W_{m-1}({\ma F}_p)^n\big|=
p^{nm}-p^{n(m-1)}=q^m-q^{m-1}$.

Now, two of these extensions $k(\vec z)$, $k(\vec w)$ satisfy
$k(\vec z)=k(\vec w)\iff \vec z=\vec j\Witt \cdot \vec w\Witt + \vec c$ with
$\vec j\in W_m({\ma F}_p)$ invertible and $\vec c\in W_m(k)$
(Proposition \ref{P8.2}). Since
$\vec z$ and $\vec w$ are $``\text{linear}"$ combinations of $\vec z_1,\ldots,
\vec z_n$ over $W_m({\ma F}_p)$, $\vec c=\vec 0$ and $\vec j\in W_m(
{\ma F}_p)^{\ast}$. Finally,
$\big|W_m({\ma F}_p)^{\ast}\big|=\big|W_m({\ma F}_p)\big|-
\big|W_{m-1}({\ma F}_p)^\big|=p^m-p^{m-1}$.

In this way, we have obtained $\frac{q^m-q^{m-1}}{p^m-p^{m-1}}$
distinct extensions $k(\vec z)/k$ of degree $p^m$ and therefore
all of them.
\end{proof}

To study how to generate this type of extensions, we have
the same result as in Theorem \ref{T7.1}. Let
$K=k(\vec y)$ be such that $\vec y^q\Witt -\vec y=\vec \alpha$ and
where $\Gal(K/k)\cong W_m(\F)$. Let
$L=k(\vec z)$ be such that $\vec z^q\Witt-\vec z=\vec \beta$.
For $\vec A_{n-1},\vec A_{n-2},\ldots, \vec A_1, \vec A_0 \in W_m(\F)$
we define $\mc R (\vec X)\in W_m(\F)[\vec X]$ by 
\[
\mc R(\vec X):=
\vec A_{n-1}\Witt \cdot \vec X^{p^{n-1}}\Witt +
\vec A_{n-2}\Witt \cdot \vec X^{p^{n-2}}\Witt +\cdots\Witt +
\vec A_{1}\Witt \cdot \vec X^{p}\Witt + \vec A_0\Witt \cdot \vec X.
\]

\begin{teorema}\label{T8.6} With the above notation, we have
$k(\vec y)=k(\vec z)$ if and only if (there exist 
$\vec A_{n-1},\vec A_{n-2},\ldots, 
\vec A_1, \vec A_0
\in W_m(\F)$ satisfying $\mc R(\vec \beta)=0$ with $\vec \beta \in 
W_m(\F) \iff \vec \beta=0$) and $\vec D\in W_m(k)$ such that 
\begin{equation}\label{Eq8.6}
\vec z=\mc R(\vec y)\Witt +\vec D.
\end{equation}
\end{teorema}

\begin{proof} The  proof is analogous to the one of Theorem
\ref{T7.2} using the formalism of Witt operations.
We will give just a few details. 

First assume that $k(\vec y)=k(\vec z)$.
Let $\sigma_i\in G$ given by
$\sigma_i(\vec y)=\vec y\Witt + \vec \mu_i$, $1\leq i\leq n$. Let
$\vec w=\mc R(\vec y)$. Then $\sigma \in G$ is given by
$\sigma=\sigma_1^{b_1}\cdots\sigma_n^{b_n}$ with $b_i\in {\ma Z}$,
$0\leq b_i\leq p^m-1$, $1\leq i\leq n$. We have
$\sigma(\vec w)=\vec w\Witt + \mc R\big(\Witt \sum_{i=1}^n \vec b_i
\Witt \cdot \vec \mu_i\big)$. In particular $\sigma_i(\vec w)=\vec w+
\mc R(\vec \mu_i)$.

If $\sigma_i(\vec z)=\vec z\Witt +\vec \xi_i$, $1\leq i\leq n$, then
$\{\vec\xi_1,\ldots,\vec\xi_n\}$ is a basis of $W_m(\F)$ over
$W_m({\ma F}_p)$.
We want to find $\vec A_{0}, \ldots,\vec A_{n-1}\in W_m(\F)$ such that
$\mc R(\vec\mu_i)=\vec \xi_i$, $1\leq i\leq n$.

We have
\begin{gather*}
\mc R(\vec\mu_i)=\vec \xi_i, 1\leq i\leq n 
\iff \vec M\Witt \cdot\left[\begin{array}{c}\vec A_0\\ \vec A_1\\ \vdots\\ \vec A_{n-2}\\
\vec A_{n-1}\end{array}\right]=\left[\begin{array}{c}\vec \xi_1\\ \vec \xi_2\\ \vdots\\ 
\vec \xi_{n-1}\\ \vec\xi_n\end{array}\right]
\end{gather*}
where $\vec M$ is the matrix
\[
\vec M=\vmatriz\mu.
\]
Now, it is clear that $\det \vec M=(\det M,\ldots, )$ where 
\[
M=\matriz{\mu_1}{\mu}.
\]
We have $\det M\in \f$ so that $\det \vec M$ is a unit of
$W_m(\F)$ so $\vec M$ is invertible. Therefore there exist
such $\vec A_i\in W_m(\F)$ and they are unique satisfying
$\sigma(\vec w)=\vec w\Witt +\vec \xi_i$. The rest of the
proof is analogous to that of Theorem \ref{T7.2}.
\end{proof}

Now we study the case of rational function fields.
Let $k=k_0(T)$ be a rational function field where
$k_0$ is a finite field such that $\F\subseteq k_0$.
We have the result analogous to that of
 \cite[Theorem 5.5]{MaRzVi2013}.

\begin{teorema}\label{T8.7} Let $K/k$ be an extension such that
$\Gal(K/k)\cong W_m(\F)$ and such that
$P_1,\ldots,P_r\in R_T^+$ and possibly $\p$, 
are the ramified primes.
Then $K=k(\vec y)$ is given by
\[
\vec y^q\Witt -\vec y=\vec \beta={\vec\delta}_1
\Witt + \cdots \Witt + {\vec\delta}_r
\Witt + \vec\gamma,
\]
with $y_1^q-y_1=\beta_1$ irreducible,
$\delta_{ij}=\frac{Q_{ij}}{P_i^{e_{ij}}}$, $e_{ij}\geq 0$, $Q_{ij}\in R_T$
and if $e_{ij}>0$, then $e_{ij}=
\lambda_{ij}p^{m_{ij}}$, $\gcd (\lambda_{ij},p)=1$,
$0\leq m_{ij}< n$, $\mcd(Q_{ij},P_i)=1$ and
$\deg (Q_{ij})<\deg (P_i^{e_{ij}})$, and $\gamma_j=f_j(T)\in R_T$ with
$\deg f_j=\nu_j p^{m_j}$ and $\gcd(q,\nu_j)=1$, $0\leq m_j<n$
when $f_j\not\in k_0$.
\end{teorema}

\begin{proof} For the first reduction of separating the irreducible
polynomials in the denominator, we proceed as in
\cite[Theorem 5.5]{MaRzVi2013}. Once we have this 
simplification, we proceed as Schmid \cite[\S 2, page 162]{Sch36}
or \cite[\S 3, page 115]{Sch36-0}, and as in the proof of
Theorem \ref{T5.2} using Corollary \ref{C8.5'}.
\end{proof}

To study the decomposition of $\p$ in a cyclic $p$--extension,
we recall the following result.

\begin{proposicion}\label{P8.8} Let $K/k$ be as in Theorem
{\rm{\ref{T8.7}}} with $n=1$. Let $\gamma_1=\cdots=
\gamma_s=0$, $\gamma_{s+1}\in k_0^{\ast}$, 
$\gamma_{s+1}\not\in \wp(k_0)$ and finally, let $t+1$ 
be the first index such that $f_{t+1}\not\in k_0$ (and
therefore $p\nmid \deg f_{t+1}$). Then, the ramification index
of $\p$ is $p^{m-t}$, the inertia degree of $\p$ is
$p^{t-s}$ and the decomposition number of $\p$ is
$p^s$. In particular $\p$ decomposes fully in
$K/k$ if and only if $\vec\gamma =\vec 0$.
\end{proposicion}

\begin{proof}
\cite[Proposition 5.6]{MaRzVi2013}.
\end{proof}

\begin{observacion}\label{O8.8'}{\rm{Another equivalent form
of Proposition \ref{P8.8}, is that $\p$ decomposes fully in $K/k$ 
if and only if
there exists $\vec \theta \in W_m(k)$ such that $\vec\gamma =\vec 
\theta^p\Witt - \theta=\wp(\theta)$.
}}
\end{observacion}

From Proposition \ref{P8.8} and with a proof analogous to that
of Proposition \ref{P6.2''}, we obtain:

\begin{proposicion}\label{P8.9} Let $K/k$ be as in Theorem
{\rm{\ref{T8.7}}}. If $\vec\gamma=\vec 0$, 
then $\p$ decomposes fully.

Conversely, if $\p$ decomposes fully there exists a decomposition
as in Theorem {\rm{\ref{T8.7}}}, with $\vec\gamma=\vec 0$.
\end{proposicion}

\begin{proof}
The proof is similar to that of Proposition \ref{P6.2}.
\end{proof}

\end{document}